\documentclass[12pt]{amsart}
\title{Algebraic K-theory of finite algebras over higher local fields}
\author{Rixin Fang}
\email{rxfang21@m.fudan.edu.cn}
\address{Shanghai Center for Mathematical Sciences, Fudan University, 2005 Songhu Road 200438, Shanghai, China}
\usepackage{amsmath, amssymb, amsthm, enumitem, comment}
\usepackage[Symbol]{upgreek}
\usepackage{hyperref} % make links be hyperlinks
\usepackage[dvipsnames]{xcolor}
\usepackage[mathscr]{euscript}
\usepackage{fullpage} 	% without this, will have wide math-article-paper margins
\usepackage{lineno}
\usepackage[normalem]{ulem}
\usepackage{tikz-cd}
\usepackage[all]{xy}
\usepackage{centernot}

\newtheorem*{claim-star}{Claim}
\newtheorem{theorem}{Theorem}[section] % numbered like the section
\newtheorem{lemma}[theorem]{Lemma}
\newtheorem{lemma-definition}[theorem]{Lemma-Definition}
\newtheorem{prop-def}[theorem]{Proposition-Definition}
\newtheorem{corollary}[theorem]{Corollary}

\newtheorem{fact-eh}[theorem]{Fact(?)}

\newtheorem{conjecture}[theorem]{Conjecture}
\newtheorem{problem}[theorem]{Problem}

\newtheorem{proposition}[theorem]{Proposition}
\newtheorem{proposition-eh}[theorem]{Proposition(?)}
\newtheorem*{theorem-star}{Theorem}
\newtheorem*{conjecture-star}{Conjecture}
\newtheorem*{question-star}{Question}
\newtheorem*{lemma-star}{Lemma}

\theoremstyle{definition}
\newtheorem{definition}[theorem]{Definition}
\newtheorem{construction}[theorem]{Construction}
\newtheorem{variant}[theorem]{Variant}
\newtheorem{convention}[theorem]{Convention}

\newtheorem{question}[theorem]{Question}

\numberwithin{equation}{section}
\newtheorem{remark}[theorem]{Remark}
\theoremstyle{remark}
\newtheorem{claim}[theorem]{Claim}

\newcommand{\Aa}{\mathbb{A}}
\newcommand{\Ff}{\mathbb{F}}

\newcommand{\Zz}{\mathbb{Z}}
\newcommand{\Nn}{\mathbb{N}}
\newcommand{\Cc}{\mathbb{C}}
\newcommand{\Pp}{\mathbb{P}}
\newcommand{\Ee}{\mathbb{E}}
\newcommand{\Ss}{\mathbb{S}}
\newcommand{\Ll}{\mathbb{L}}

\newcommand{\calC}{\mathcal{C}}

\newcommand{\calO}{\mathcal{O}}

\newcommand{\Rom}[1]{\uppercase\expandafter{\romannumeral #1\relax}}

\newcommand{\calA}{\mathcal{A}}
\newcommand{\scrS}{\mathscr{S}}

\newcommand{\suspen}{\Sigma^{\infty}}
\newcommand{\modulo}{\operatorname{mod}}
\newcommand{\Map}{\operatorname{Map}}
\newcommand{\map}{\operatorname{map}}
\newcommand{\Hom}{\operatorname{Hom}}
\newcommand{\cofib}{\operatorname{cofib}}
\newcommand{\fib}{\operatorname{fib}}
\newcommand{\fil}{\operatorname{Fil}}
\newcommand{\gr}{\operatorname{gr}}
\newcommand{\ext}{\operatorname{Ext}}
\newcommand{\tor}{\operatorname{Tor}}
\newcommand{\Alg}{\operatorname{Alg}}
\newcommand{\CAlg}{\operatorname{CAlg}}
\newcommand{\Mod}{\operatorname{Mod}}

\newcommand{\coh}{\operatorname{H}}
\DeclareMathOperator{\desc}{desc}
\DeclareMathOperator{\can}{can}
\DeclareMathOperator{\motfil}{fil_{mot}^*}

\DeclareMathOperator{\Fun}{Fun}
\DeclareMathOperator{\id}{id}

\DeclareMathOperator{\Sp}{Sp}
\DeclareMathOperator{\MU}{MU}
\DeclareMathOperator{\BP}{BP}
\DeclareMathOperator{\HH}{HH}
\DeclareMathOperator{\K}{K}
\DeclareMathOperator{\THH}{THH}
\DeclareMathOperator{\TP}{TP}
\DeclareMathOperator{\TR}{TR}
\DeclareMathOperator{\TC}{TC}

\newenvironment{claimproof}[1][\proofname]
               {
                 \proof[#1]
                 
               }
               {
                 \endproof
               }

\begin{document}
\begin{abstract}
It is known that the truncated Brown--Peterson spectra can be equipped with a certain nice algebra structure, by the work of J. Hahn and D. Wilson, and these ring spectra can be viewed as rings of integers of local fields in chromatic homotopy theory. Furthermore, they satisfy both Rognes' redshift conjecture and the Lichtenbaum--Quillen property. For lower-height cases, the K-theory of the truncated polynomial algebras over these ring spectra is  well understood through the work of L. Hesselholt, I. Madsen, and others. In this paper, we demonstrate that the Segal conjecture fails for truncated polynomial algebras over higher chromatic local fields, and consequently, the Lichtenbaum--Quillen property fails. However, the weak redshift conjecture remains valid. Additionally, we provide some other examples where Segal conjecture holds.
\end{abstract}
\maketitle
\tableofcontents
\section{Introduction}
\subsection{Motivation}\label{sec1.1}
Algebraic K-theory is an important invariant for rings and more generally ring spectra. It captures some arithmetic information of the input ring, and it reflects chromatic information of the input ring spectrum. Furthermore, it is not merely an invariant for ring spectra, but also an universal localizing invariant in the sense of \cite{BGT13}. It also can output the sphere spectrum, which serves as the initial object in stable homotopy theory.

To be more precise about the information that algebra K-theory captures, we let $\calO$ be a $\Zz[1/\ell]$-algebra, where $\ell$ is a prime number. Quillen \cite{Quillenconj} conjectured that there is a spectral sequence \begin{equation*}
      E_2^{*,*}= \coh_{\textup{\'{e}t}}^*(\calO, \Zz/\ell^n(*)) \Longrightarrow
    \K_*(\calO, \Zz/\ell^n)
\end{equation*}
that converges to the target in large degrees. Lichtenbaum \cite{Lichtenbaum} made a conjecture relate algebraic K-theory groups to special values of zeta functions. Wiles proved that the quotients of \'etale cohomology groups of $\calO$ recover special values of zeta functions for certain nice rings \cite{Wiles}. So, the conjecture made by Lichtenbaum coincides with Quillen's conjecture for certain nice rings. The author learned this from \cite{Rognes1999}.

If moreover, $\calO$ contains an $\ell^n$-th root of unity, 
then Thomason \cite{ThomasonetaleK} proved that there is a spectral sequence:
\begin{equation*}
    E_2^{p,q} =\begin{cases}
        \coh^{p}_{\textup{\'et}}(\calO, \Zz/\ell^n(i)), q=2i,\\
        0, q \textup{~odd}.
    \end{cases}
    \Longrightarrow \K_{q-2p}(\calO, \Zz/\ell^n)[\beta^{-1}].
\end{equation*}
Here, the element $\beta\in \K_2(\calO, \Zz/l^n)$ corresponds to an $\ell^n$-th root of unity. In this case, Bott inverted K-theory coincides with topological K-theory (defined in Thomason's work); furthermore, in sufficiently large degrees, all three (Bott inverted, topological, and \'etale) become equivalent. It has been  showed that the Bott inverted K-theory is equivalent to $L_1^f \K(\calO)\otimes \Ss/\ell \simeq \K(\calO)\otimes \Ss/\ell[v_1^{-1}] \simeq L_{K(1)} \K(\calO)\otimes\Ss/\ell$ for $\ell>2$, see \cite[Appendix A]{ThomasonetaleK}.
Moreover, when $\calO$ does not contain $\ell^n$-th roots of unity, there is a spectral sequence with the same $E_2$-page that converges to the \'etale K-theory. 

Thus, the conjecture reduces to comparing \'etale K-theory with K-theory, and comparing $\K(\calO, \Zz/\ell)$ with $L_1^f(\K(\calO ))\otimes \Zz/\ell \simeq L_{K(1)} \K(\calO,\Zz/\ell)$. In fact, there is a motivic spectral sequence
\begin{equation*}
   E_2^{p,q} =\coh^{p-q}_{\textup{mot}} (X, \Zz/\ell(-q))\Longrightarrow \K_{-p-q}(X,\Zz/\ell),
\end{equation*}
where $X$ is a smooth scheme over a field $k$. This can be proved using motivic homotopy theory. Within this framework, motivic cohomology is represented by the motivic Eilenberg\textendash Maclane spectrum, and the spectral sequence above is an analogue of the Atiyah\textendash Hirzebruch spectral sequence. 
If moreover, $\ell \in k^{\times}$, then the $E_2$-page above is equivalent to $\coh^p_{\textup{\'et}} (X, \Zz/\ell(q))$ for all $p\leq q\leq 0$ (norm-residue theorem), see \cite{BlochKato}. Waldhausen reformulated the Lichtenbaum\textendash Quillen conjecture as follows. 

\begin{conjecture}[{\cite{Waldhausenconj}}]\label{Waldhausen}
For a nice ring $\calO$, the map  \begin{equation*}
 \K(\calO,\Zz/\ell) \to L_1^f \K(\calO)\otimes \Ss/\ell\simeq \K(\calO, \Zz/\ell)[v_1^{-1}] \simeq L_{K(1)} (\K(\calO))\otimes \Ss/\ell   
\end{equation*} is an equivalence on large degrees.     
\end{conjecture}

The original Lichtenbaum\textendash Quillen conjecture is related to special values of the zeta function, a central topic in number theory. In contrast, Waldhausen’s reformulation of the conjecture adopts a more topological perspective, fitting naturally within chromatic homotopy theory. The K-theory of certain nice rings exhibit $v_1$-periodicity, because the localization is almost an equivalence. A natural question is whether there exists a higher level chromatic analogue. This question is addressed by Rognes' red-shift conjecture. There are actually many forms of this conjecture, the key point of which it that, for a ring spectrum exhibiting $v_n$-periodicity, the K-theory of it should exhibit $v_{n+1}$-periodicity. 

In studying these phenomena, we begin by fixing a prime number $p$. A spectrum $X$ is said to be fp spectrum, if it is $p$-complete, bounded below, and there is a $p$-local finite spectrum $F$ such that $F\otimes X$ is $\pi$-finite (bounded, and the homotopy groups are finite groups). We say $X$ has fp-type  $n$, if $n=\min\{k-1\mid F\otimes X ~\text{is} ~\pi\text{-finite}, F~\text{is fp-type}~k\}$. A spectrum $X$ is said to be pure fp-type $n$, if moreover $F_*X$ is a free finitely generated $P(v)$-module for some finite CW-spectrum of type $n$, with $v_n$-map $v: \Sigma^d F\to F$. In his 2000 Oberwolfach lecture \cite{2000oberwolfachlecture}, Rognes introduced the term "redshift", and the redshift problem is stated as follows.
\begin{problem}[{Chromatic red-shift problem, see \cite[page 8]{2000oberwolfachlecture}}]\label{redshiftprob}
Let $R$ be an $S$-algebra of pure fp type $n$. Does $\TC(R,p)$ have pure fp type $n+1$?
\end{problem}
Related to Problem \ref{redshiftprob}, there is a conjecture stated as follows.
\begin{conjecture}[{\cite[page 14-15]{Guido},\cite[page 1278]{HW22}}]\label{conj:red1}
For suitable $\Ee_1$-rings $R$ of fp-type $n$, $\K(R)^{\wedge}_p$ is of fp-type $n+1$. 
\end{conjecture}

By Theorem \ref{thm:LQproperty}, Conjecture \ref{conj:red1} will be exactly the analogue of Conjecture \ref{Waldhausen}.
In \cite{HW22}, the authors proved that $\BP\langle n\rangle^{\wedge}_p$ satisfies Conjecture \ref{conj:red1}, for any $n\geq -1$. And this also gives a positive answer to Problem \ref{redshiftprob}, in some cases.

In chromatic homotopy theory, $v_n$ is detected by Morava K-theory, $K(n)$. One could use Morava K-theory to define the height for spectra, although, in general, this is not a very nice concept. We say an $\Ee_{\infty}$-ring $R$ is of height $n$, if $L_{K(n)} R\not\simeq 0$, and $L_{K(n+1)} R\simeq 0$. 
Hahn \cite{heightforEinfty} showed that for an $\Ee_{\infty}$-ring, if $L_{K(i)} R\simeq 0$, then $L_{K(j)} R\simeq 0$, for all $j\geq i$. Thus, in this case, the height is well defined. Let $R$ be a ring spectrum. So, one might define the height of $R$ to be the number $n$ such that $K(n)_*R\ne 0$ and $K(m)_*R=0$ for $m>n$. 
\begin{conjecture}\label{conj:red2}
Let $R$ be a ring spectrum of height $n$. Then $K(R)$ is  of height $n+1$.
\end{conjecture}
The conjecture is true for $R$ any $\Ee_{\infty}$-ring, by \cite{exampleschromaticredshift},  \cite{chromaticnullstellensatz}, \cite{CMNN}, \cite{purity}.
For the relation of the first conjecture to $T(n)$-localized theory, see Theorem \ref{thm:LQproperty}. The second conjecture is more related to $K(n)$-localized theory. As known, for $n\geq 2$, \cite{telescope} shows the telescope conjecture fails, using many ingredients in algebraic K-theory and things related to these two conjectures.

The method to prove Conjecture \ref{conj:red1} was presented in \cite{HW22} and also appeared in \cite{AN21}. It mainly reduces to the following two questions.

\begin{question}[Segal conjecture]\label{q:Segal}
Let $R$ be a nice ring spectrum of fp type $n$.
Let $F$ be a $p$-local type $n+2$ finite complex. Does the Frobenius map \begin{equation*}
    F_*(\THH(R)) \to F_*(\THH(R)^{t C_p})
\end{equation*} 
induce isomorphism in large degrees?
\end{question}

\begin{question}[Weak canonical vanishing]\label{q:can}
Let $R$ be a nice ring spectrum of fp type $n$. Let $F$ be a $p$-local type $n+2$ finite complex. Is there a $d\in \Zz$, such that the maps \begin{equation*}
    F_*(\THH(R)^{h C_{p^k}}) \to F_*(\THH(R)^{t C_{p^k}})
\end{equation*} are trivial for $*>d$, and for any $k\geq 1$?
\end{question}

Since the boundedness of $\TR$ would imply the boundedness of $\TC$, one can make following definition.
\begin{definition}\label{def:LQproperty}
Let $F$ be a $p$-local type $n+2$ finite complex,  and $R$ be a ring spectrum of fp type $n$. We say that $R$ satisfies the Lichtenbaum\textendash Quillen property if $F\otimes\TR(R)$ is bounded.
\end{definition}
With this definition, one could also ask the following question related to Problem \ref{redshiftprob}.
\begin{problem}\label{prob:LQforTR}
    Let $R$ be an $\Ee_1$-ring spectrum of (pure) fp type $n$. Does $R$ satisfy the Lichtenbaum--Quillen property?
\end{problem}
The method to prove Conjecture \ref{conj:red2}, is first to construct a nice algebra map $R\to S$, yielding an algebra map $L_{K(n+1)} \K(R) \to L_{K(n+1)} \K(S)$. If one can show that $L_{K(n+1)} \K(S)\not\simeq 0$, then one could conclude that the source is not $0$. 

The K-theory of truncated polynomial algebras over perfect fields were studied in \cite{HM97}. Note that $\Ff_p[x]/x^e$ is just $\BP\langle -1\rangle[x]/x^e$. Thus it is natural to consider the K-theory of $\BP\langle n\rangle[x]/x^e$, which are chromatic higher level cases of $\Ff_p[x]/x^e$.
The goal of this paper is to study Problem \ref{redshiftprob},  Conjecture \ref{conj:red1}, \ref{conj:red2}, Question \ref{q:Segal}, \ref{q:can}, for $R=\BP\langle n\rangle[x]/x^e$, and other rings related to that.  In particular, in the proof of Proposition \ref{pro:segal1} and \cite[Theorem A]{HM97}, shows that $\K(\Ff_p[x]/x^e)$ is not an fp spectrum of type $0$. And Proposition \ref{pro:TClxe}, shows that $\K((\ell[x]/x^e)^{\wedge}_p)^{\wedge}_p$ is not an fp spectrum of type $2$.
\subsection{Main results}
The chromatic redshift problem \ref{redshiftprob} does not assume regularity hypothesis, i.e. does not require $R$ to be "regular" in some sense. In this paper, we show that a regularity hypothesis is in fact necessary, see Theorem \ref{thmE}. The truncated Brown–Peterson spectra may be regarded, in a suitable sense, as regular ring spectra, in light of Theorem \ref{thmC} and \ref{thmD}.

Using the functor $\desc_{\Ff_p}$, we can approximate the Frobenius map of $\THH(\BP\langle n\rangle[x])$ by its associated graded, thus proving the following.
\begin{theorem}[Proposition \ref{pro:polyBPNSegal}]\label{thmC}
Let $F$ be any type $n+1$ finite complex, then the Frobenius\begin{equation*}
    F_*(\THH(\BP\langle n\rangle[x])) \to F_*(\THH(\BP\langle n\rangle[x])^{t C_p})
\end{equation*} 
induces isomorphism on large degrees. In particular, if $F$ is of type $n+2$, the above also holds. If one replaces $\BP\langle n\rangle$ by $\BP\langle n\rangle \otimes \Ss_{W(k)}$ the statement holds still. Here, $k$ is a perfect field of characteristic $p$, $\Ss_{W(k)}$ is the spherical Witt vectors.
\end{theorem}

As a consequence, and combining the $\Aa^1$-invariance of K-theory, we are able to prove the following.
\begin{theorem}[Proposition \ref{cor:LQforpoly}]\label{thmD}
Let $F$ be any $p$-local finite type $3$ complex. Then the following spectra are bounded.
\begin{align*}
&F\otimes \TC(\BP\langle 1\rangle[x]),\\
&F\otimes \K(\BP\langle 1\rangle[x])^{\wedge}_p.
\end{align*}
Consequently, the following maps have bounded above fibers
\begin{align*}
&\TC(\BP\langle 1\rangle[x])\to L_2^f (\TC(\BP\langle 1\rangle[x])),\\
&\K(\BP\langle 1\rangle[x])^{\wedge}_p \to L_2^f(\K(\BP\langle 1\rangle[x])^{\wedge}_p).
\end{align*}
\end{theorem}
In the classical case, algebraic K-theory has the $\Aa^1$-invariance property for regular rings. This result shows that algebraic K-theory has $\Aa^1$-invariance property for regular ring spectra in the following sense. If $R$ is a regular ring spectrum such that $F\otimes \K(R)$ is bounded for a certain ($p$-local) finite spectrum $F$, then $F\otimes \K(R[x])$ is still bounded.

By decomposing $\THH(\Ss[x]/x^e)$ into a sum of $S^1$-spectra, and analyzing the Frobenius map, we found the following theorem. Actually, the base case is Proposition \ref{pro:segal1}.
\begin{theorem}[see Proposition \ref{pro:segalfails}, \ref{cor:TRboundedfails}, \ref{thm:generaltruncatedpoly}]\label{thmE}
Let $F$ be any $p$-local type $n+2$ finite complex. Let $W$ denote the fiber of Frobenius map of $F\otimes\THH(\BP\langle n\rangle)$. The following statements are true.
\begin{enumerate}
    \item If $W\not\simeq 0$, then the fiber of the Frobenius \begin{equation*}
    F\otimes \THH(\BP\langle n\rangle[x]/x^e) \to F\otimes \THH(\BP\langle n\rangle[x]/x^e)^{t C_p}
\end{equation*}
is not bounded above, where $e>1$.  In other words, $F$-Segal conjecture for $\THH(\BP\langle n\rangle[x]/x^e)$ is true if and only if $W\simeq 0$.

\item If $F$ is a sum of generalized Moore complexes $\Ss/(p^{i_0}, v_1^{i_1}, \ldots v_{n+1}^{i_{n+1}})$, then $W\not\simeq 0$.
\item In fact, for $F$ a $p$-local type $n+2$ finite complex, $W\not\simeq 0$.
\item If $e>1$, the spectrum $F\otimes \TR(\BP\langle n\rangle[x]/x^e)$ is not bounded. In other-words, the Lichtenbaum--Quillen property defined in Definition \ref{def:LQproperty} fails for $\BP\langle n\rangle[x]/x^e$.

\item Let $S$ be a ring spectrum of fp type $n$, $F$ be a $p$-local finite $n+2$ finite complex. We denote the fiber of the Frobenius of $F\otimes \THH(S)$ by $V$. Then $F$-Segal conjecture for $\THH(S[x]/x^e)$ is true if and only if $V \simeq 0$. Here, $e>0$.
\end{enumerate}
\end{theorem}
From the above theorem, one would not expect Problem \ref{redshiftprob} to have a positive answer in general, because $F\otimes\TR(\BP\langle n\rangle[x]/x^e,p)$ is not bounded. In other words, Problem \ref{prob:LQforTR} has a negative answer. Moreover, by Proposition \ref{pro:TClxe}, we know that $\TC(\BP\langle 1\rangle[x]/x^e, p)$ is not an fp spectrum.
So, in the spectral context, one could argue that algebraic K-theory remains sensitive to singularities.
By the work in \cite{fptype1}, in height $1$, $\BP\langle1\rangle[x]/x^e$ are simplest examples of pure fp spectra violating the Lichtenbaum\textendash Quillen property, since as spectra, these spectra are just sums of $\BP\langle 1\rangle$, and $\BP\langle 1\rangle$ generates the thick sub-category of fp spectra of type $1$.
In fact, this phenomenon happens in general. See Proposition \ref{pro:segalfailsforgroupring}, Corollary \ref{cor:segalfailsforgroupring}.
However, we found that if one replaces $\BP\langle n\rangle$ with $\BP$ or $\MU$, the Segal conjecture holds, i.e. we have following theorem. 
\begin{theorem}[Proposition \ref{pro:segalforgroupring}, \ref{pro:segalfortruncated}]
    The Frobenius maps 
\begin{align*}
\THH(\MU[x]/x^e) \to \THH(\MU[x]/x^e)^{t C_p},\\
\THH(\BP[x]/x^e) \to \THH(\BP[x]/x^e)^{t C_p},&\\
\THH(\MU[\Omega S^3]) \to \THH(\MU[\Omega S^3])^{t C_p},&\\
\THH(\BP[\Omega S^3]) \to \THH(\BP[\Omega S^3])^{t C_p}.&
    \end{align*}
    are $p$-adic equivalences. In general, for $Y$ a space satisfying conditions (1), (2), (3) in Proposition \ref{pro:segalfailsforgroupring}, then the Frobenius maps \begin{gather*}
       \THH(\MU[\Omega Y]) \to \THH(\MU[\Omega Y])^{t C_p},\\
        \THH(\BP[\Omega Y]) \to \THH(\BP[\Omega Y])^{t C_p}.
    \end{gather*}
    are $p$-adic equivalences.
\end{theorem}
\subsection{Relation to other works}
\begin{remark}
This is not the first paper to exhibit failures of
the Lichtenbaum--Quillen property for fp ring spectra. There are several papers in the literature doing this, implicitly in many old papers such as in \cite{DGM13}, and more explicitly in \cite{Lee-Levy} and \cite{HLS24}. Nevertheless, the examples in this paper appear to be relatively simpler than those in previous works.
\end{remark}
\begin{remark}
The term "higher local fields" first appeared in \cite{2localfield} (the chromatic height one case).
The author suspected that the truncated Brown--Peterson spectra should be thought as the valuation rings of "higher local fields", a phrase that also appears in the title of this paper. In other words, the truncated Brown--Peterson spectra may be thought as "higher local number rings". 
After the completion of this work, Devalapurkar--Hahn--Rognes defined the notion "higher local number ring", see \cite{higherarithematicduality}.
\end{remark}

\subsection{Future work}
It seems that one could use the method presented in this paper to investigate some other examples of the Segal conjecture, by using the cyclic decomposition, i.e. decomposition of $\THH(\Ss[M])$ or $\THH(\Ss[\Omega Y])$, to carry out more computations, where $M$ is a commutative monoid, $Y$ is a space. For instance, in a forthcoming work,  we aim to compute $\TC(\BP\langle n\rangle[x,y]/(xy))$ at low heights. Similarly, on could replace $\BP\langle n\rangle$ by $K(n)$ or $k(n)=\tau_{\geq 0} K(n)$.

It should be also noted that, there are some similarities between this paper and \cite{HLS24}. In that paper, the authors use the decomposition of $\THH$ for square zero extensions to compute mod $(p, v_1^{p^{n-2}})$ syntomic cohomology of $\Zz/p^n \simeq \BP\langle 0\rangle/(v_0)^n$. The key point is that they found that taking an animated ring $R$ to its mod $(p, v_1^{p^n})$ syntomic cohomology factors through the construction $R\mapsto R/p^{n+2}$. It would be interesting to see this phenomenon extended to more general ring spectra.

In particular, one could try to explore the following question, first by considering $\BP\langle n\rangle$.
\begin{question}[{\cite[Question 1.23]{HLS24}}]
Does mod $(p, v_1, \ldots ,v_{n+1})$ syntomic cohomology for height $n$ ring spectra, factor though its mod $(p^{i_0}, v_i^{i_1}, \ldots ,v_n^{i_n})$ reduction for some sequence $i_0, \ldots,i_n$ ?
\end{question}

Another interesting question is as follows.

\begin{question}
Is the telescope conjecture true for $\TC(\BP\langle n\rangle[x]/x^e)$. Here, $n\geq 1, e>1$. In other words, is the natural map
\begin{equation*}
    L_{n+1}^f \TC(\BP\langle n\rangle[x]/x^e) \to L_{n+1} \TC(\BP\langle n\rangle[x]/x^e)
\end{equation*}
an equivalence?
\end{question}

\subsection{Structure of the paper}
In Section \ref{sec2}, we use several $\Ee_{\infty}$-ring maps to construct several $\Ee_3$-algebras over some base ring spectra other than $\Ss$, and we also use the tower of $\Ee_{\infty}$-algebras in \cite[\S 2]{HW22} to get an $\Ee_3$-algebra $\BP\langle n\rangle/v_n^e$. All these algebras are finite algebras over $\BP\langle n\rangle$, see Proposition \ref{finalgebra}.
In Section \ref{sec3}, we use several $\Ee_3$-algebra maps to prove Conjecture \ref{conj:red2}, and to prove Theorem \ref{thmB}. In Section \ref{sec4}, we use the descent functor $\desc_{\Ff_p}$ to prove the Segal conjecture, and we use the decompositions of $\THH(\Ss[x]/x^e), \THH(\Ss[\Omega S^3])$ to prove the failure of the Segal conjecture for some ring spectra. The primary goal is to prove Theorem \ref{thmC}, \ref{thmD}, \ref{thmE}. In Section \ref{sec5}, we use the evenly free map $\THH(\MU[x])\to \MU[x]$ to compute the $E_2$-page of the descent spectral sequence which computes $\THH(\BP\langle n\rangle[x]/x^e; \Ff_p[x]/x^e)$. We show that the spectral sequence degenerates at the $E_2$-page in some cases. See Proposition \ref{pro:descentE2}. In Section \ref{sec6}, we use the decomposition of $\THH(\Ss[x]/x^e)$ try to carry out some computations for $V(2)_*\TC(\ell[x]/x^e)$. We prove Proposition \ref{pro:TClxe}.
\subsection{Notations and conventions}
We use the language of $\infty$-categories developed in \cite{htt}, and the theory of higher algebra developed in \cite{ha}, for convenience. We denote the category of spectra by $\Sp$, and category of spaces by $\scrS$. And for an $\Ee_n$-monoidal category $\calC$, $n\geq 2$, the category of $\Ee_m$-algebras over $\calC$ is denoted by $\Alg_{\Ee_m}(\calC)$, where $0\leq m< n$. For $R\in \Alg_{\Ee_n}(\Sp)$, we use $\Alg^{(m)}_R$ to denote $\Alg_{\Ee_m} (\Mod_R)$. When $m=1,$ we write $\Alg_R^{(1)}$ as $\Alg_R$, when $m=\infty$, we write $\Alg_R^{(\infty)}$ as $\CAlg_R$. We use the symbol $\otimes^{\Ee_m}_R$ to indicate the tensor product in $\Alg^{(m)}_R$. When $R$ is the unit $\Ss\in \Sp$, we simply omit $R$. We refer to an object in $\Alg^{(m)}_R$ as an $\Ee_m$-$R$-algebra. When $m=1$, we simply refer to an object in $\Alg_R$ as an $R$-algebra. When $m=\infty$, we refer to an object in $\CAlg_R$ as a commutative $R$-algebra. When $R=\Ss$, we simply omit the word '$R$', or we call it $\Ee_m$-ring. 
A (discrete) ring simply indicate a classical ring. Let $R$ be a ring, we use $\otimes^{\Ll}$ to denote the derived tensor product. We use $\Lambda_R(x_1,\ldots, x_n)$ to denote the exterior algebra over $R$ generated by $x_1,\ldots ,x_n$. And we use $\Gamma_R(x_1,\ldots, x_n)$ to denote the divided power algebra generated by $x_1, \ldots ,x_n$, which sometimes we also write as $R\langle x_1,\ldots,x_n\rangle$. For a stable $\infty$-category $\calC$ with a $t$-structure, we use $\tau_{\leq n}, \tau_{\geq n}$ to denote the truncation functors. When $\calC=\Sp$, $X\in \Sp_{\leq n}$, if the localization $X\to \tau_{\leq n} X$ is an equivalence, and $X\in \Sp_{\geq m}$, if $\tau_{\geq m} X\to X$ is an equivalence. When $X\in \Sp_{\leq n}$, for some $n$, we say $X$ is bounded above, and $X\in \Sp_{\geq m}$, for some $m$, we say that $X$ is bounded below. If there are $a,b$ such that $X\in \Sp_{[a,b]}=\Sp_{\geq a}\cap \Sp_{\leq b}$, then $X$ is said to be bounded. Let $\calC$ be a symmetrical monoidal category, and $A\in \Alg(\calC)$. We use \cite[Definition \Rom{3}.2.3]{NS18}, to define $\THH(A)$, for instance, $A\in \Alg$, $\THH(A)=A\otimes_{A\otimes A^{\textup{op}}}A$.
\subsection{Acknowledgments}
The author would like to thank his advisor Guozhen Wang for carefully reading the early drafts. The author would like to thank Chris Brav for helping to improve the writing of this article. The author would like especially to thank Lars Hesselholt, for suggesting the author Proposition \ref{pro:segal1}, and telling the author Proposition \ref{pro:decomposation}. The author would like to thank Gabriel Knoll for detailed feedback. The author would like to thank Oliver Braunling and John Rognes for comments. Finally, the author thanks Jingbang Guo and Wei Yang, for useful conversations. The author is partially supported by Grant NSFC-12226002. The author sincerely thanks the IAS/PCMI Summer School and IWoAT for their inspiring programs. The author benefited greatly from these two wonderful conferences, during which part of this work was completed.
\section{Truncated polynomials and infinitesimal thickening}\label{sec2}
The main goal of this section is to precisely describe the algebras will be concerned in the next sections. Basically all of them should be viewed as finite algebras over $\BP\langle n\rangle$, see Proposition \ref{finalgebra}. We fix a prime number $p$ for the remaining context. First, we recall the definitions and results in \cite[\S 2]{HW22}.
\begin{definition}
    Let $1\leq k\leq \infty$ and $n\geq 0.$ Let $B$ be a $p$-local $\Ee_k$-MU-algebra. We say that $B$ is an $\Ee_k$-MU-algebra form of $\BP\langle n\rangle$ if the composite 
    $$
    \Zz_{(p)}[v_1,v_2,\ldots, v_n] \subseteq \BP_*\subseteq \pi_*\MU_{(p)} \to \pi_* B
    $$
is an isomorphism.
\end{definition}
\begin{theorem}[{\cite[Theorem 2.0.6]{HW22}}]\label{e3form}
    For $n\geq -1,$ there exists an $\Ee_3$-$\MU$-algebra form of $\BP\langle n\rangle$.
\end{theorem}
\begin{construction}\label{trunpoly}
We define $\Ss[x]:=\suspen_+ \Nn.$ In a same fashion, we may consider the truncated version.
    For $e\in \Nn_{\geq 1} $,  $M_e:=\{1,x,\ldots, x^{e-1}\}\simeq \Nn/([e,\infty)\cap\Nn)$ where the multiplication is defined in an obvious way.
    The multiplication of $M_e,\Nn$ are strictly commutative. Unwinding the definitions in \cite[Definition 2.4.2.1, Remark 2.4.2.2]{ha}, we learn that $\Nn, M_e$ are commutative monoid objects in $\scrS$, and hence commutative algebra objects in $\scrS$. 
Let $\Ss[x]/x^e := \suspen_+ M_e \in \Sp.$  There are $\Ee_{\infty}$-ring maps \begin{gather*}
    \Ss[x] \to \Ss[x]/x^e ,\\
    \Ss[x]/x^e \to \Ss.
\end{gather*}
induced by maps of commutative monoid  \begin{gather*}
    \Nn \to \Nn/([e,\infty)\cap \Nn)\\
    n \mapsto [n],\\
    \Nn/[e,\infty)\to \Nn/([1,\infty)\cap \Nn)\\
    [1] \mapsto [0].
\end{gather*}
Then for any spectrum $A \in \Sp,$ we can define the polynomial spectrum and truncated polynomial spectra over $A$ as $A[x]:= A\otimes \Ss[x],$ $A[x]/x^e:= A\otimes \Ss[x]/x^e.$
In generally, one can also define spectral polynomials of multiple variables. Concretely,
\begin{gather*}
    A[x_1,x_2\ldots, x_n]:= A[x_1, x_2, \ldots, x_{n-1}]\otimes \Ss[x_n], \\ 
    A[x_1,x_2,\ldots,x_n]/(x_1^{e_1}, \ldots ,x_n^{e^n}):= A[x_1,x_2,\ldots,x_{n-1}]/(x_1^{e_1}, \ldots ,x_{n-1}^{e^{n-1}})\otimes \Ss[x_n]/x_n^{e_n}.
\end{gather*}
\end{construction}
Now, we take $A$ to be $\MU.$ Since $\suspen_+: \scrS\to \Sp$ is a symmetric monoidal functor \cite[Proposition 4.8.2.18, Corollary 4.8.2.19]{ha}, we know that $\Ss[x], \Ss[x]/x^e$ are commutative algebras in $\Sp.$ Recall that $\MU$ is an $\Ee_{\infty}$-ring, and we have a symmetric monoidal functor $\Mod_{\Ss} \to \Mod_{\MU}$ given by tensor with $\MU$, see \cite[Proposition 7.1.2.7, Proposition 4.6.2.17]{ha}. We thus conclude that $\MU[x], \MU[x]/x^e$ are commutative algebras in $\Mod_{\MU},$ and by \cite[Variant 7.1.3.8]{ha}, $\MU[x], \MU[x]/x^e$ are also commutative algebras in $\Sp$. 

As a spectrum $\Ss[x]$ is nothing but a direct sum of $\Ss$, and this is also true for $\Ss[x]/x^e,$ using the multiplication on $\Nn, M_e,$ we can get \begin{gather*}
    \MU[x]_*\simeq  \MU_*[x],\\
    \pi_*\MU[x]/x^e\simeq \MU_*[x]/x^e.
\end{gather*}
\begin{variant}
    In general, let $M$ be a commutative monoid, one can define $A[M]:=A\otimes \suspen_+ M$, for any spectrum $A.$ For example, $M$ is a sub monoid of $\Nn^{\times m}$, or a quotient monoid of $\Nn^{\times m}.$ 
    In a upcoming paper, we will consider $\BP\langle n\rangle [x, y]/(xy)$. 
\end{variant}
\begin{definition}\label{defe3form}
    Let $1\leq k\leq \infty$ and $n\geq 0$. Let $B$ be a $p$-local $\Ee_k$-$\MU[x]/x^e$-algebra. We say that $B$ is an $\Ee_k$-$\MU[x]/x^e$-algebra form of $\BP\langle n\rangle [x]/x^e$ if the composite 
    \begin{equation*}
        \Zz_{(p)}[v_1,v_2,\dots, v_n] [x]/x^e\subseteq \BP_*[x]/x^e\subseteq \pi_*\MU_{(p)}[x]/x^e\to \pi_*B
    \end{equation*}
is an isomorphism. By convention, $\BP\langle -1\rangle:= \Ff_p$, and $\BP\langle -1\rangle[x]/x^e$ is $\Ff_p[x]/x^e$.
\end{definition}
\begin{definition}
    Let $1\leq k\leq \infty$ and $n\geq 0$. Let $B$ be a $p$-local $\Ee_k$-$\MU[x]$-algebra. We say that $B$ is an $\Ee_k$-$\MU[x]$-algebra form of $\BP\langle n\rangle [x]/x^e$ if the composite 
    \begin{equation*}
        \Zz_{(p)}[v_1,v_2,\dots, v_n] [x]\subseteq \BP_*[x]\subseteq \pi_*\MU_{(p)}[x]\to \pi_*B
    \end{equation*}
is a surjection and has kernel generated by $x^e$.
\end{definition}
We can view $\BP\langle n\rangle^{\wedge}_p$ as height $n$ version of $\Zz_p$, we should think $(\BP\langle n\rangle[x]/x^e)^{\wedge}_p$ as height $n$ version of $\Zz_p[x]/x^e$. Putting all the things above together, we have the following fact.
\begin{proposition}\label{e3form2}
    There is an $\Ee_3$-algebra $A$ such that $A$ is the image of 
   an $\Ee_3$-$\MU[x]$-algebra form of $\BP\langle n\rangle[x]/x^e$ under the forgetful functor.
\end{proposition}
\begin{proof}
By the definition of $\MU[x]$-algebra form of $\BP\langle n\rangle[x]/x^e$, it suffice to show there is an $\Ee_3$-$\MU[x]/x^e$-algebra form of $\BP\langle n\rangle[x]/x^e$.  This is because we have an $\Ee_{\infty}$-ring map $\MU[x]\to \MU[x]/x^e$ by Construction \ref{trunpoly}, and according to \cite[Proposition 7.1.2.7]{ha} we know the forgetful functor $\Mod_{\MU[x]/x^e}\to \Mod_{\MU[x]}$ is lax symmetric monoidal. 

According to Theorem \ref{e3form}, we get an object in $\Alg_{\MU}^{(3)}$. By Construction \ref{trunpoly} we learn that there is a map of commutative algebras $\MU\to \MU[x]/x^e$, so there is a symmetric monoidal functor $\Mod_{\MU}\to \Mod_{\MU[x]/x^e}$. We thus get an $\Ee_3$-$\MU[x]$-algebra $B[x]/x^e$, and it is clear that $B[x]/x^e$ satisfies the condition in Definition \ref{defe3form}.  Now, let $A$ be the image of $B[x]/x^e$ under the lax symmetric monoidal functor $\Mod_{\MU[x]} \to \Mod_{\Ss}.$
\end{proof}
\begin{remark}\label{remarke3alg1}
We let $A$ denote the $\Ee_3$-algebra in Proposition \ref{e3form2} and let $B$ denote the underlying $\Ee_3$-algebra of the $\Ee_3$-$\MU$-algebra $\BP\langle n\rangle$ constructed in Theorem \ref{e3form}.  We know that \begin{equation*}
 (B\otimes^{\Ee_3}\Ss[x]/x^e)\simeq A
\end{equation*} as spectra, this essentially is \cite[Proposition 4.6.2.17]{ha}.
To show that the equivalence is an equivalence of $\Ee_3$-algebras, it suffice to produce an $\Ee_3$-algebra map $B \otimes^{\Ee_3}\Ss[x]/x^e\to A$ which induce isomorphism as spectra, see \cite[Remark 7.1.1.8]{ha}. However, this  \textbf{would not} be an equivalence of $\Ee_3$-algebras in general, unless, $\BP\langle n\rangle$ admits an $\Ee_{\infty}$-$\MU$-algebra structure. But we can still say a bit more.
\begin{claim}\label{remarkclaim1}
    We have an $\Ee_3$-algebra map $B \otimes\Ss[x]/x^e \to A$, and we have an equivalence of $\Ee_{\infty}$-algebra $B\otimes \Ss[x]/x^e\to A$ when $\BP\langle n\rangle
    $ admits an $\Ee_{\infty}$-$\MU$-algebra structure.
\end{claim}
Note that we have following two adjunction pairs \begin{gather*}
 f^*:\Mod_{\Ss}\leftrightarrows \Mod_{\MU}:f_*,\\
 g^*:\Mod_{\MU}\leftrightarrows \Mod_{\MU[x]/x^e}:g_*.
\end{gather*}
induced from the $\Ee_{\infty}$-ring maps $f:\Ss\to \MU, g:\MU\to \MU[x]/x^e$.
So $B=f_*(\BP\langle n\rangle), A\simeq f_* g_*(g^*(\BP\langle n\rangle))$. The second equivalence follows from following facts\begin{itemize}
    \item We have $\Ee_{\infty}$-ring maps $\Ss\to \MU \to \MU[x]\to \MU[x]/x^e$.
    \item Computing the forgetful functor $\Mod_{\MU[x]}\to \Mod_{S}$ is equivalent to compute the composite
    $\Mod_{\MU[x]/x^e} \to \Mod_{\MU}\to \Mod_{\Ss}$.
\end{itemize}
\begin{claim}\label{remarkclaim2}
    We have an equivalence of $\Ee_{\infty}$-$\MU$-algebras $\MU\otimes\Ss[x]/x^e\simeq f^*(\Ss[x]/x^e)$.
\end{claim}
\begin{claimproof}
Note that for an commutative algebra $R$, $\CAlg_R$ may identifies with $\CAlg_{R/}$, see \cite[Variant 7.1.3.8]{ha}. 
Now note that  the left adjoint of the forgetful functor  $\CAlg_{\MU}\to \CAlg_{\Ss}$ is given by \begin{equation*}
\xymatrix{\CAlg_{\Ss}\simeq \CAlg_{\Ss/}\ar[r]^{\sqcup \MU} &\CAlg_{\MU/}\simeq \CAlg_{\MU}},
\end{equation*} see \cite[Proposition 5.2.2.8]{htt}. Hence,
$f^*(\Ss[x]/x^e)\simeq \Ss[x]/x^e\sqcup \MU \simeq \Ss[x]/x^e\otimes \MU$.
\end{claimproof}
\begin{claim}\label{remarkclaim3}
We have an equivalence of $\Ee_3$-rings $A \simeq f_*(\BP\langle n\rangle \otimes_{\MU}^{\Ee_3} \MU[x]/x^e)$.
\end{claim}
\begin{claimproof}
Because $A\simeq f_* g_*(g^*(\BP\langle n\rangle))$.
It suffice to prove that $g_*g^*(\BP\langle n\rangle)\simeq \BP\langle n\rangle \otimes^{\Ee_3}_{\MU} \MU[x]/x^e$ is  an equivalence of $\Ee_3$-$\MU$-algebras, which is clearly an equivalence of $\MU$-modules. We only need to produce an $\Ee_3$-$\MU$-algebra map $\BP\langle n\rangle \otimes^{\Ee_3}_{\MU}\MU[x]/x^e \to  g_*g^*(\BP\langle n\rangle)$. According to \cite[Proposition 7.1.2.6]{ha}, \cite[Proposition 5.2.2.8]{htt}, we have a natural transformation of lax symmetrical monoidal functors $\id \to g_*g^*$,  hence we have an adjunction pair  \begin{equation*}
    g^*: \Alg_{\Ee_3}(\Mod_{\MU}) \leftrightarrows \Alg_{\Ee_3}(\Mod_{\MU[x]/x^e}): g_*.
\end{equation*}
Hence we have an equivalence of mapping spaces $\Map_{\Alg^{(3)}_{\MU[x]/x^e}}(g^* c, d)\simeq \Map_{\Alg^{(3)}_{\MU}}(c, g_*d)$. Now take $c= \BP\langle n\rangle \otimes^{\Ee_3}_{\MU}\MU[x]/x^e, d=g^*(\BP\langle n\rangle)$. Since $g^*$ is symmetric monoidal, thus $g^*(\BP\langle n\rangle \otimes_{\MU}^{\Ee_3} \MU[x]/x^e)\simeq g^*(\BP\langle n\rangle)\otimes^{\Ee_3}_{\MU[x]/x^e} \MU[x]/x^e \simeq g^*(\BP\langle n\rangle)$, therefore we have an equivalence of $\Ee_3$-$\MU$-algebras $\BP\langle n\rangle \otimes^{\Ee_3}_{\MU}\MU[x]/x^e \simeq g_*g^*(\BP\langle n\rangle)$.
\end{claimproof}
\begin{claimproof}[Proof of Claim \ref{remarkclaim1}]
Now, we need to produce an $\Ee_3$-algebra map 
\begin{equation*}
   f_*(\BP\langle n\rangle)\otimes^{\Ee_3}\Ss[x]/x^e \to f_*(\BP\langle n\rangle\otimes^{\Ee_3}_{\MU} \MU[x]/x^e). 
\end{equation*}
Since $f_*$ is lax symmetric monoidal, and according to \cite[Proposition 7.1.2.7]{ha}, \cite[Proposition 5.2.2.8 (2)]{htt}, Claim \ref{remarkclaim2}, \ref{remarkclaim3}, we have $\Ee_3$-algebra maps \begin{equation*}
f_*(\BP\langle n\rangle)\otimes^{\Ee_3} \Ss[x]/x^e\to f_*(\BP\langle n\rangle)\otimes^{\Ee_3} f_*f^* \Ss[x]/x^e \to f_*(\BP\langle n\rangle
 \otimes^{\Ee_3}_{\MU} \MU[x]/x^e)\simeq A.
\end{equation*}
If $\BP\langle n\rangle$ has an $\Ee_{\infty}$-$\MU$-algebra structure, we may write $\BP\langle n\rangle \otimes^{\Ee_{\infty}}_{\MU} \MU[x]/x^e$ as a coproduct in $\CAlg_{\MU}\simeq \CAlg_{\MU/}$, see \cite[Proposition 3.2.4.7, Variant 7.1.3.8]{ha}. Hence 
$f_*(\BP\langle n\rangle \otimes_{\MU}^{\Ee_{\infty}}\MU[x]/x^e)$ may identifies as the push-out $B \sqcup_{\MU}(\MU \sqcup \Ss[x]/x^e)$ in $\CAlg_{\Ss}\simeq \CAlg_{\Ss/}$ hence equivalent to $B\sqcup \Ss[x]/x^e \simeq B\otimes \Ss[x]/x^e$.
\end{claimproof}

\end{remark}
\begin{convention}\label{conven1}
 Hence, as an $\Ee_3$-algebra we may write $A$ as $\BP\langle n\rangle \otimes_{\MU}\MU[x]/x^e$,
 and we will denote $B\otimes \Ss[x]/x^e$ as $\BP\langle n\rangle[x]/x^e$. When $\BP\langle n\rangle$ admits an $\Ee_{\infty}$-$\MU$-algebra structure or $e=0$ we will not distinguish these two notations. We denote the $\Ee_m$-$\MU[x]/x^e$-algebra and the $\Ee_m$-$\MU[x]$-algebra in Proposition \ref{e3form2} as $\BP\langle n\rangle[x]/x^e$, where $0\leq m\leq 3$.
\end{convention}

\begin{remark}\label{differentforms}
One can also use the strategy in \cite[\S 2]{HW22} to construct an $\MU[x]$-algebra form of $\BP\langle n\rangle[x]/x^e$, i.e. by computing $\mathscr{U}^{(m)}_{\MU[x]}(\BP\langle n\rangle[x]/x^e)$. It seems that one could only get an $\Ee_1$-form, due to computational difficulties. We have following commutative diagram.
\begin{equation*}
    \xymatrix{
   \BP\langle n\rangle \in \ar[d] \Alg^{(1)}_{\MU} \ar[d] &\ar[l]_= \Alg^{(1)}_{\MU} \ar[d]  \ni \BP\langle n\rangle \ar[d] & \ar[l] \Alg^{(3)}_{\MU} \ar[d]  \ni \BP\langle n\rangle \ar[d] \\
  \BP\langle n\rangle[x]\in \Alg^{(1)}_{\MU[x]} \ar[d]  &\ar[d]\Alg^{(1)}_{\MU[x]/x^e}  \ni \BP\langle n\rangle[x]/x^e & \ar[l] \Alg^{(3)}_{\MU[x]/x^e}  \ni \BP\langle n\rangle [x]/x^e \ar[d] \\
    \BP\langle n\rangle[x]/x^e \in \Alg^{(1)}_{\MU[x]} \ar[dr]& \ar[l] \ar[d]  \Alg^{(1)}_{\MU[x]/x^e} \ni \BP\langle n\rangle[x]/x^e & \Alg^{(3)}_{\MU[x]}\ar[dl] \ni \BP\langle n\rangle[x]/x^e \\ 
    & \BP\langle n\rangle\otimes_{\MU}\MU[x]/x^e\in\Alg & 
}
\end{equation*}
\end{remark}
\begin{construction}\footnote{The author observed this construction from \cite[\S 5]{HLS24}.}\label{exterior}
   For any spectrum $A,$ We define $\Lambda_A(x)$ to be the spectrum $A\otimes (\Ss\otimes_{\Ss[x]}\Ss)$ where $\Ss$ view as an $\Ss[x]$-module via Construction \ref{trunpoly} when $e=1$. If moreover, $A \in \Alg_{\Ee_n},$ where $1\leq n\leq \infty$. By the same argument, this gives rise an $\Ee_n$-algebra $\Lambda_A(x)$ in $\Sp.$  For instance, $\MU\in \CAlg,$ then $\Lambda_{\MU}(x)$ is also a commutative algebra.  Using the Tor spectral sequence, one computes the homotopy group of $\Lambda_{\Ss}(x)$ as $\Lambda_{\pi_*\Ss}(x),$ where $|x|=1,$ see for instance \cite[Proposition 3.6]{Ang08}. We learn that $\pi_*(\Lambda_{\Ss}(x))$ is a free $\pi_*(\Ss)$-module. Then if $A\in \Alg_{\Ee_2}$, we have isomorphism $\pi_*\Lambda_A(x)\simeq \Lambda_{\pi_*A}(x)$.
\end{construction}
\begin{definition}
       Let $1\leq k\leq \infty$ and $n\geq 0.$ Let $B$ be a $p$-local $\Ee_k$-$\Lambda_{\MU}(x)$-algebra. We say that $B$ is an $\Ee_k$-$\Lambda_{\MU}(x)$-algebra form of $\Lambda_{\BP\langle n\rangle}(x)$ if the composite 
    \begin{equation*}
        \Zz_{(p)}[v_1,v_2,\dots, v_n]\otimes \Lambda(x)\subseteq \pi_*\Lambda_{\BP}(x)\subseteq \pi_*\Lambda_{\MU}(x)\to \pi_*B
    \end{equation*}
is an isomorphism. By convention $\Lambda_{\BP\langle -1\rangle}(x):=\Lambda_{\Ff_p}(x)$.
\end{definition}
\begin{proposition}\label{e3form3}
    There is an $\Ee_3$-algebra $A$, and $A$ is the image of an $\Ee_3$-$\Lambda_{\MU}(x)$-algebra form of $\Lambda_{\BP\langle n\rangle}(x)$ under the forgetful functor. 
\end{proposition}
\begin{proof}
    This is clear, the proof is similar to the proof of Proposition \ref{e3form2}.
\end{proof}
\begin{remark}\label{remarke3alg2}
Similar to Remark \ref{remarke3alg1},  we have the following  claim. \begin{claim}
We denote the underlying $\Ee_3$-algebra of $\BP\langle n\rangle$ as $B$, and $A$ is denoted to be the algebra in Proposition \ref{e3form3}. We have an $\Ee_3$-algebra map 
\begin{equation*}
    B\otimes \Lambda_{\Ss}(x) \to A.
\end{equation*}
There is an equivalence of commutative algebras $B\otimes \Lambda_{\Ss}(x) \simeq A$, when $\BP\langle n\rangle$ has an $\Ee_{\infty}$-$\MU$-algebra structure.
\end{claim}
\end{remark}
\begin{convention}\label{conven2}
Hence as an $\Ee_3$-algebra,  we may write $A$ as 
$\BP\langle n\rangle\otimes_{\MU} \Lambda_{\MU}(x)$, we will denote $B \otimes \Lambda_{\Ss}(x)$ to be $\Lambda_{\BP\langle n\rangle}(x)$. When $\BP\langle n\rangle$ admits an $\Ee_{\infty}$-$\MU$-structure we will not distinguish these two notations. We denote the $\Ee_m$-$\Lambda_{\MU}(x)$-algebra in Theorem \ref{e3form2} as $\Lambda_{\BP\langle n\rangle}(x)$, where $0\leq m\leq 3$.    
\end{convention}
Note that $v_n^e \in \pi_* \BP\langle n\rangle,$ we can form the cofiber $\BP\langle n\rangle/v_n^e$. By \cite[Theorem 1.4]{Bur22} and Theorem \ref{e3form}, we find that $\BP\langle n\rangle/v_n^e$ can be equipped with an $\Ee_2$-algebra structure, as long as $e>2$. In fact, the strategy in \cite[\S 2]{HW22} says that we can constructed an $\Ee_3$-$\MU[y]/y^e$-algebra form of $\BP\langle n\rangle/v_n^e$. We will have a similar diagram as in Remark \ref{differentforms}.
\begin{definition}
    Let $1\leq k\leq \infty$ and $n\geq 0.$ Let $B$ be a $p$-local $\Ee_k$-$\MU$-algebra. We say that $B$ is an $\Ee_k$-$\MU$-algebra form of $\BP\langle n \rangle/v_n^e$ if the composite 
    \begin{equation*}
        \Zz_{(p)}[v_1,v_2,\dots, v_n]\subseteq \pi_*\BP\subseteq \pi_*\MU\to \pi_*B
    \end{equation*}
    is a surjection and the kernel is generated by $v_n^e$, for $e\geq 1$. 
\end{definition}
In order to prove there is an $\Ee_3$-$\MU$-algebra form of $\BP\langle n\rangle/v_n^e$, we need to recall the results in \cite[\S 2]{HW22}, and briefly sketch the strategy. In \cite[\S 2]{HW22}, the authors constructed a tower of $\Ee_{\infty}$-$\MU$-algebras \begin{equation}\label{tower}
  \MU[y]\to \cdots  \to \MU[y]/y^k\to \MU[y]/y^{k-1}\to \cdots \to \MU[y]/y^2\to \MU
\end{equation}
where the $y$ lies in degree $2j\in 2\Nn_{>0}$. The commutative algebra $\MU[y]$ is very different from the Construction \ref{trunpoly} unless $j=0$. We construct the $\Ee_3$-algebra form of $\BP\langle n\rangle$ by induction on $n.$ When $n=-1,0$, the $\Ee_3$-form (actually $\Ee_{\infty}$) of $\BP\langle n\rangle$ can be constructed via Postnikov truncation and note that \cite[Theorem 7.1.3.15]{ha}. Suppose we have an $\Ee_3$-$\MU$-algebra form of $\BP\langle n\rangle$, we will inductively construct an $\Ee_3$-$\MU[y]/y^k$-algebra form $B_k$ using the square zero extension $\MU[y]/y^k\to \MU[y]/y^{k-1}$ and some information of the cotangent complex, where $|y|=2(p^{n+1}-1)$. More precisely, we have the following fact.
\begin{proposition}[{\cite[Proposition 2.6.2]{HW22}}]\label{square0}
Fix $n\geq 0$, and let $B$ to be an $\Ee_3$-$\MU$-algebra form of $\BP\langle n \rangle$, then there exists a sequence of maps 
\begin{equation*}
    \cdots \to B_k\to B_{k-1} \to \cdots B_2\to B_1=B
\end{equation*}
satisfying following properties.
\begin{enumerate}
    \item $B_k$ is given an $\Ee_3$-$\MU[y]/y^k$-algebra structure.
    \item Each map $B_k\to B_{k-1}$ is an $\Ee_3$-$\MU[y]/y^k$-algebra map.
    \item  The map $B_k\to B_{k-1}$ induce an equivalence of 
    $\Ee_3$-$\MU[y]/y^{k-1}$-algebras \begin{equation*}
        \MU[y]/y^{k-1} \otimes_{\MU[y]/y^k} B_k\simeq B_{k-1}.
    \end{equation*}
    \item  The map $B\to \cofib(B_2\to B_1)\simeq \Sigma^{|y|+1} B$ is detected by $\delta v_{n+1}$ in \begin{equation*}
        E_2=E_{\infty}\simeq \ext^{*,*}_{\MU_*}(B_*, B_*)\Longrightarrow \pi_*\map_{\MU}(B,B)
    \end{equation*} 
\end{enumerate}
\end{proposition}
\begin{proposition}\label{e3form4}
    There is an $\Ee_3$-algebra $A$, and $A$ is the image of an $\Ee_3$-$\MU$-algebra form of $\BP\langle n\rangle/v_n^e$ under the forgetful functor, for any $n\in \Nn$ and $e\geq 1$. 
\end{proposition}
\begin{proof}
By Theorem \ref{e3form}, we take $B$ to be an $\Ee_3$-$\MU$-algebra form of $\BP\langle n \rangle$, and then by Proposition \ref{square0}, we get a tower\begin{equation*}
    B_e\to B_{e-1} \to \cdots \to B_2\to B
\end{equation*}
 satisfies the conditions in Proposition \ref{square0}. By the construction, or properties in Proposition \ref{square0} we know that $\cofib(B_k\to B_{k-1})\simeq \Sigma^{k|y|}B_{k-1}.$
 Therefore the associated graded has homotopy group as
 \begin{equation*}
     \pi_*\gr B_e \simeq B_*[y]/y^e. 
 \end{equation*}
 Note that the tower in (\ref{tower}) has the similar graded 
 \begin{equation*}
     \pi_* \gr \MU[y]/y^e\simeq \MU_*[y]/y^e.
 \end{equation*}
 Therefore by our assumption on $B$, we know that
 \begin{equation*}
     \Zz_{(p)}[v_1,\ldots,v_n][y]/y^e\subseteq \MU_*[y]/y^e \to \pi_*B_e
\end{equation*}
 is an isomorphism.
 According to Proposition \ref{square0}, we know that the image of $y \in \pi_*\MU[y]/y^2$ in $\pi_*B_2$ is the same as the image $v_{n+1}\in \MU_* \to \pi_* B_2 \simeq B_*[y]/y^2$. Thus $v_{n+1} \in \MU_*$ maps to $y$ in $\pi_* B_e\simeq B_*[y]/y^e$, so the map $\Zz_{(p)}[v_1,v_2,\ldots, v_{n+1}] \subseteq \MU_* \to \pi_*\MU[y]/y^e \to \pi_*B_e$ is a surjection and has kernel generated by $v_{n+1}^e$. Now we let $A$ be the image of $B_e$ under the forgetful functor $\Mod_{\MU}\to \Mod_{\Ss}$.
\end{proof}
\begin{convention}\label{conven3}
    Sometimes, we will denote $\BP\langle n\rangle/v_n^e$ constructed above as $\BP\langle n,e\rangle$ for simplicity.
\end{convention}
\begin{construction}
 Let $\MU[y]/y^e$ be the $\Ee_{\infty}$-$\MU$-algebra (can also be viewed as a graded commutative $\MU$-algebra) in the tower (\ref{tower}), where $|y|=2j \in 2\Nn_{>0}$, and we also have an $\Ee_{\infty}$-ring map $\MU\to \MU[y]/y^e$.
And let $B$ be the form constructed in Proposition \ref{e3form4}, then $B\otimes_{\MU}\MU[y]/y^e$ gives arise an $\Ee_3$-$\MU[y]/y^e$-algebra. We will denote it to be $B[y]/y^e$, then using the $\Ee_{\infty}$-map $\MU[y]\to \MU[y]/y^e$, one gets an $\Ee_3$-$\MU[y]$-algebra $B[y]/y^e$, where $|y|=2j$. This construction is different from Proposition \ref{e3form2}, unless $j=0$. Since $\MU[y]/y^e$ is free as $\MU$-module, thus one computes $\pi_*(B[y]/y^e)\simeq \BP\langle n\rangle_*[y]/y^e$, where $|y|=2j$.
\end{construction}
We therefore have proved following fact.
\begin{proposition}\label{e3form5}
    There is an $\Ee_3$-algebra $A$, and $A$ is the image of an $\Ee_3$-$\MU[y]$-algebra form $B$ of $\BP\langle n\rangle[y]/y^e$ under the forgetful functor. The following map
    \begin{equation*}
        \Zz_{(p)}[v_1,v_2,\ldots,v_n][y]\subseteq \pi_*\MU[y]\to B_*
    \end{equation*}
is a surjection and the kernel is generated by $y^e$.
\end{proposition}
\begin{construction} Let $\Lambda_{\MU}(y):=\MU\otimes_{\MU[y]}\MU$, since the $\MU\to \MU[y]$ is an $\Ee_{\infty}$ map, therefore $\Lambda_{\MU}(y)$ is an $\Ee_{\infty}$-ring. One computes the homotopy group $\pi_*\Lambda_{\MU}(y)\simeq \Lambda_{\MU_*}(y)$, where $|y|=2j+1$. With the same argument before, we can construct an $\Ee_3$-$\Lambda_{\MU}(y)$-algebra form of $\Lambda_{\BP\langle n\rangle}(y)$.\end{construction}
\begin{remark}\label{remarkkernel}
    By carefully checking the argument of constructing an $\Ee_3$-$\MU$-algebra form of $\BP\langle n\rangle$ and Proposition \ref{square0}, one can actually prove that, there is an $\Ee_3$-$\MU$-algebra form $B$ of $\BP\langle n\rangle$ such that the map \begin{equation*}
      \pi_*\MU_{(p)}\to \pi_* B
    \end{equation*}
is a surjection, and has kernel generated by $\{v_m, x_j| m\geq n+1, \forall k , j\ne p^k-1\}$. And the similar statement is also hold for $\Ee_3$-$\MU$-form of $\BP\langle n\rangle/v_n^e$, and $\Ee_3$-$\MU[x]$-form $\BP\langle n\rangle[x]/x^e$.
\end{remark}
So far, we have constructed several types of $\Ee_3$-algebras.
Basically, all of them can be viewed as truncated polynomial algebras over $\BP\langle n\rangle$ for some $n$, but the polynomial generators may not live in the degree $0$. The form in Proposition \ref{e3form2}, $\BP\langle n\rangle[x]/x^e$ is just a finite direct sum of $\BP\langle n\rangle$ (note that $\BP\langle n\rangle[x]/x^e\simeq \BP\langle n\rangle_{\MU}\MU[x]/x^e$ as spectra). The form in Proposition \ref{e3form3} is equivalent to $\BP\langle n\rangle \bigoplus \Sigma^1 \BP\langle n\rangle$ as $\BP\langle n\rangle$-module. The form in Proposition \ref{e3form4} is finite as $\BP\langle n\rangle$-module. Because by construction, the cofiber $\cofib(B_k\to B_{k-1}) \simeq \Sigma^{k(2p^{n+1}-2)} \BP\langle n\rangle$, thus $\BP\langle n+1\rangle/v_{n+1}^e$ is a finite $\BP\langle n\rangle$-module. Since $\MU[y]/y^e$ is finite as $\MU$-module, we therefore conclude that the forms in Proposition \ref{e3form5} are finite as $\BP\langle n\rangle$-modules. We thus have proved following fact.
\begin{proposition}\label{finalgebra}
    Each form in Proposition \ref{e3form2}, \ref{e3form3}, \ref{e3form4}, \ref{e3form5}, is a finite $\BP\langle n\rangle$-module for some $n$. Thus we can view them as finite algebras over $\BP\langle n\rangle$ (the ring of integers of a height $n$ local field).
\end{proposition}
In the following, we will use the conventions in Convention \ref{conven1}, \ref{conven2}, we will use the name of the form constructed in the Proposition \ref{e3form4}, \ref{e3form5}. For instance, we will refer to $\BP\langle n\rangle/v_n^e$ as both an algebra and an $\MU$-algebra in Proposition \ref{e3form4}.  We now present some computational results of these algebras, before that, we present a well known lemma.
\begin{lemma}\label{dividedpower}
    Let $R$ be an ordinary commutative ring. We have following equivalence of graded rings:
    \begin{equation*}
        \THH_*((R[x]/x^e)/R[x])\simeq \Gamma_{R[x]/x^e}(\sigma^2 x^e).
    \end{equation*}
\end{lemma}
\begin{proof}
 We need to compute $\THH(R[x]/x^e/R[x])\simeq \HH((R[x]/x^e)/R[x])$. To this end, we first compute $\mathscr{U}^{(1)}_{R[x]}(R[x]/x^e)\simeq R[x]/x^e\otimes_{R[x]}^{\Ll} R[x]/x^e$. We can resolve $R[x]/x^e$ by $\Lambda_{R[x]}(\sigma x^e)$\footnote{The notation $\sigma$ can be defined using factorization homology which we will update this later.} as an $R[x]$-module, this is because $x^e$ is not a zero-divisor. Therefore $\mathscr{U}_{R[x]}^{(1)}(R[x]/x^e) \simeq \Lambda_{R[x]/x^e}(\sigma x^e)$, this can be also obtained by using \cite[Proposition 3.6]{Ang08}. Then, we compute $R[x]/x^e\otimes_{\Lambda_{R[x]/x^e}(\sigma x^e)}^{\Ll} R[x]/x^e$. As a $\Lambda_{R[x]/x^e}(\sigma x^e)$-algebra, $R[x]/x^e$ is equivalent to $\Gamma_{\Lambda_{R[x]/x^e}(\sigma x^e)}(\sigma^2 x^e)$ which can be also written as $\Lambda_{R[x]/x^e}(\sigma x^e)\langle\sigma^2 x^e\rangle$. We therefore have $\HH_*((R[x]/x^e)/R[x])\simeq (R[x]/x^e)\langle \sigma^2 x^e\rangle$. 
\end{proof}
\begin{proposition}
The commutative algebra $\THH((\Ff_p[x]/x^e)/\MU[x])$ has homotopy ring equivalent to
    \begin{equation*}
    P_{\Ff_p}\otimes_{\Ff_p}\Gamma_{\Ff_p[x]/x^e}(\sigma^2 x^e),
    \end{equation*}
where $P_{\Ff_p}$ is a polynomial algebra over $\Ff_p$, and one of the generators is detected by $\sigma^2 p$. In particular, $\THH((\Ff_p[x]/x^e)/\MU[x])$ is even.
\end{proposition}
\begin{proof}
        By the multiplicative property of THH for $\Ee_{\infty}$-rings, we have following equivalences:
\begin{align*}
 \THH((\Ff_p[x]/x^e)/\MU[x]) & \simeq\THH(\Ff_p/\MU)\otimes \THH((\Ss[x]/x^e)/\Ss[x]),\\
~ & \simeq \THH(\Ff_p/\MU)\otimes_{\Ff_p} (\Ff_p
\otimes\THH((\Ss[x]/x^e)/\Ss[x])),\\
~ &  \simeq\THH(\Ff_p/\MU)\otimes_{\Ff_p}\THH((\Ff_p[x]/x^e)/\Ff_p[x])).
\end{align*}
All the equivalence above are equivalence of commutative algebras.
 We view $\THH(\Ff_p/\MU)$ as $\Ee_{\infty}$-$\Ff_p$-algebra via $\Ee_{\infty}$-algebra maps $\Ff_p\to \THH(\Ff_p)\to\THH(\Ff_p/\MU)$. Then the second equivalence right above follows from following facts.\begin{itemize}
      \item For an $\Ee_{\infty}$-algebra $R$, we have equivalence of categories $\CAlg_R\simeq \CAlg_{R/}$, see \cite[Variant 7.1.3.8]{ha}.
     \item The symmetrical monoidal structure on $\CAlg_R$ is coCartesian, see \cite[Proposition 3.2.4.7]{ha}, hence computing tensor product is the same as computing push-out.
 \end{itemize}
 Now the desire statement follows from Lemma \ref{dividedpower} and \cite[Proposition 2.4.2]{HW22}.
\end{proof}
\begin{proposition}
The commutative algebra $\THH((\BP\langle 0\rangle[x]/x^e)/\MU[x])$ has homotopy ring as   
\begin{equation*}
    \Zz_{(p)}[w_{1,i}: i\geq 0]\otimes_{\Zz_{(p)}} \Zz_{(p)}[y_{j,i}: j\not\equiv-1\modulo p, i\geq 0] \otimes_{\Zz_{(p)}} \Gamma_{\Zz_{(p)}[x]/x^e}(\sigma^2 x^e).
    \end{equation*}
\end{proposition}
In particular $\THH(\BP\langle 0\rangle[x]/x^e/\MU[x])$ is even.
\begin{proof}
The proposition follows from the same argument as above, together with \cite[Proposition 2.5.3]{HW22}, and note that $\HH_*(\Zz_{(p)}[x]/x^e/\Zz_{(p)}[x])$ is free as $\Zz_{(p)}$-module.
\end{proof}
\begin{proposition}\label{THHtruncated}
    There is a spectral sequence \begin{equation*}
        E^2\Longrightarrow \THH_*((\BP\langle n\rangle[x]/x^e)/\MU[x]),
    \end{equation*} whose $E^{\infty}$-page is given by\begin{equation*}
        E^{\infty}\simeq \Gamma_{\BP\langle n\rangle_*[x]/x^e}(\sigma x^e, \sigma^2 v_i: i\geq n+1)\otimes_{\BP\langle n\rangle_*[x]/x^e}\Gamma_{\BP\langle n\rangle_*[x]/x^e}(\sigma^2 x_j: \forall k\in\Nn, j\ne p^k-1).
    \end{equation*}
In particular, $\THH(\BP\langle n\rangle[x]/x^e/\MU[x])$ is even.
\end{proposition}
\begin{proof}
    We need to compute $\BP\langle n\rangle[x]/x^e\otimes_{\mathscr{U}^{(1)}_{\MU[x]}(\BP\langle n\rangle[x]/x^e)} \BP\langle n\rangle [x]/x^e.$ By \cite[Proposition 3.6]{Ang08} and Remark \ref{remarkkernel}, \begin{equation*}
        \pi_*(\mathscr{U}^{(1)}_{\MU[x]}(\BP\langle n\rangle[x]/x^e))\simeq \Lambda_{\BP\langle n\rangle_*[x]/x^e}(\sigma x^e, \sigma v_i: i\geq n+1)\otimes \Lambda(\sigma x_j: j\ne p^k-1).
    \end{equation*}
Then the Tor spectral sequence has the formula:
\begin{equation*}
    E^2\simeq E^{\infty}\simeq \Gamma_{\BP\langle n\rangle_*[x]/x^e}(\sigma^2 x^e, \sigma^2 v_i, \sigma x_j: i\geq n+1, j\ne p^k-1)\Longrightarrow \pi_*(\THH((\BP\langle n\rangle[x]/x^e)/\MU[x]). 
\end{equation*}

Since every thing is concentrated in even degree, thus $\THH(\BP\langle n\rangle[x]/x^e/\MU[x])$ is even.
\end{proof}
\begin{remark}
Let $A_*=\BP\langle n\rangle_*[x]/x^e$.
We can construct a ring map 
\begin{equation*}
A_*[w_{n+1,i},y_{j,i}: i\geq 0, j\not\equiv p-1 \modulo p] \otimes_{A_*} A_*[\gamma_{p^m}(\sigma^2 x^e): m\geq 0]\to \THH_*(\BP\langle n\rangle[x]/x^e/\MU[x])
\end{equation*}
by choosing lifts of classes $\gamma_{p^i}(\sigma^2 v_{n+1}), \gamma_{p^i}(\sigma^2 x_j), \gamma_{p^i}(\sigma^2 x^e)$ in the Tor spectral sequence.
Since the fiber of  $\BP\langle n\rangle[x]/x^e\to \BP\langle 0\rangle[x]/x^e$  lies in $\Mod_{\MU[x]}^{\geq |v_1|}$, we conclude that the fiber of $\THH(\BP\langle n\rangle[x]/x^e/\MU[x])\to \THH(\Zz_{(p)}[x]/x^e/\MU[x])$ lies in $\Mod_{\MU[x]}^{\geq |v_1|}$. Assume $p<|v_1|=2p-2$, then $\THH_p(\Zz_{(p)}[x]/x^e/\MU[x])\simeq \THH_p(\BP\langle n\rangle [x]/x^e/\MU[x])$, then according to the functoriality of Tor spectral sequence, we conclude that $p\gamma_1^p(\sigma^2 x^e)= \gamma_p(\sigma^2 x^e)$.  
\end{remark}

\begin{conjecture}
From the above remark, we conjecture that for $n\geq 0, p>2$ the ring spectrum $\THH(\BP\langle n\rangle[x]/x^e/\MU[x])$ has homotopy ring as
\begin{equation*}
 A_*[w_{n+1,i},y_{j,i}: i\geq 0, j\not\equiv p-1 \modulo p] \otimes_{A_*} A_*[1/p^m(\sigma^2 x^e)^{p^m}: m\geq 0],   
\end{equation*}
i.e. the ring map in the remark is an surjection and the kernel is generated by $\{p^m \gamma_1(\sigma^2 x^e)-\gamma_{p^m}(\sigma^2 x^e)\}$. This is true for $n=0$.     
\end{conjecture}

\begin{proposition}\label{THHexterior}
The homotopy of the ring spectrum
    $\THH(\Lambda_{\BP\langle n\rangle}(x)/\Lambda_{\MU}(x))$
is isomorphic to (as a graded ring)   \begin{equation*}
\BP\langle n\rangle_*[w_{n+1,i}, y_{j,i}]\otimes \Lambda(x).
    \end{equation*}
\end{proposition}
\begin{proof}
Let $f:\MU\to \Lambda_{\MU}(x)$ be the $\Ee_{\infty}$-ring map in Construction \ref{exterior}, then $f^*:\Mod_{\MU}\to \Mod_{\Lambda_{\MU}(x)}$ is symmetrical monoidal by \cite[Proposition 7.1.2.7]{ha}. Now from the definition of $\THH$ in \cite[Definition \uppercase\expandafter{\romannumeral3}.2.3]{NS18}, we have following equivalences:
\begin{align*}
\THH(\Lambda_{\BP\langle n\rangle}(x)/\Lambda_{\MU}(x))& \simeq \THH(f^*(\BP\langle n\rangle)/f^*(\MU))\\
& \simeq f^*(\THH(\BP\langle n\rangle/\MU) \\
& \simeq \THH(\BP\langle n\rangle/\MU)\otimes_{\MU}\Lambda_{\MU}(x).\\
\end{align*}
The second equivalence is an equivalence of $\Ee_2$-$\Lambda_{\MU}(x)$-algebras,
the third equivalence is an equivalence of $\Ee_2$-algebras,
hence according to Construction \ref{exterior} we have \begin{equation*}
\pi_*\THH(\Lambda_{\BP\langle n\rangle}(x)/\Lambda_{\MU}(x))\simeq \BP\langle n\rangle_*[w_{n+1,i}, y_{j,i}]\otimes \Lambda(x),\end{equation*} where $|x|=1$.
\end{proof}
\begin{proposition}\label{THHquotient}
The ring spectrum $\THH(\BP\langle n\rangle/v_n^e/\MU)$ is even.
\end{proposition}
\begin{proof}
Using \cite[Proposition 3.6]{Ang08}, we have
\begin{equation*}
\pi_*\mathscr{U}^{(1)}_{\MU}(\BP\langle n\rangle/v_n^e)\simeq \Lambda_{\BP\langle n\rangle_*/v_n^e}(\sigma v_n^e, \sigma v_i: i\geq n+1) \otimes \Lambda_{\BP\langle n\rangle_*/v_n^e}(\sigma x_j: j\ne p^k-1).
\end{equation*} 
Then the Tor spectral sequence has the formula:
\begin{equation*}
    E^2\simeq E^{\infty}\simeq \Gamma_{\BP\langle n\rangle_*/v_n^e}(\sigma^2 v_n^e, \sigma^2 v_i, \sigma x_j: i\geq n+1, j\ne p^k-1)\Longrightarrow \pi_*(\THH((\BP\langle n\rangle/v_n^e)/\MU). 
\end{equation*}
Since everything is concentrated in even degree, therefore $\THH(\BP\langle n\rangle/v_n^e/\MU)$ is even.
\end{proof}
\section{Negative topological cyclic homology}\label{sec3}
The primary goal of this section is to investigate chromatic localized algebraic K-theory and also the $v_{n+1}$-periodicity of $\TC^-$. In particular, we prove  Theorem 
\ref{thmB}. These were also studied in \cite{AR02}, \cite{HW22}. We denote $K(n)$ to be the $n$th Morava K-theory.
\begin{definition}
A spectrum $A$ is height $n$ if $K(h)_*A\ne 0$ and $K(h+k)_*A=0$ for $k>0$.
\end{definition}
\begin{proposition}\label{redshift1}
With the same notations in Section \ref{sec2}, the spectrum $\BP\langle n\rangle[x]/x^e\simeq \BP\langle n\rangle_{\MU} \MU[x]/x^e$ has height $n$. And 
\begin{gather*}
 L_{K(n+1)}\K(\BP\langle n\rangle[x]/x^e)\not\simeq 0,\\
 L_{K(n+1)}(\K(\BP\langle n\rangle \otimes_{\MU}\MU[x]/x^e))\not\simeq 0.  
\end{gather*}
\end{proposition}
\begin{proof}
Since $\BP\langle n\rangle[x]/x^e$ is complex orientated, hence according to \cite[Chapter 6, Proposition 4.1]{tmf} the localization can be computed as $(v_m^{-1} \BP\langle n\rangle[x]/x^e)^{\wedge}_{(p,v_1,\ldots, v_{m-1})}$. Therefore for any $ m\geq n+1$, $L_{K(m)}\BP\langle n \rangle [x]/x^e\simeq 0$ by the fact $v_m$ maps to $0$ in homotopy, see Remark \ref{remarkkernel}. According to Claim \ref{remarkclaim1}, there is an $\Ee_3$-algebra map $\BP\langle n\rangle[x]/x^e\to \BP\langle n\rangle\otimes_{\MU} \MU[x]/x^e$. Now by  Claim \ref{remarkclaim2}, there is an $\Ee_3$-algebra map $\BP\langle n\rangle\otimes_{\MU} \MU[x]/x^e \to\BP\langle n\rangle$
induced by an $\Ee_3$-$\MU$-algebra map $\MU[x]/x^e\to \MU$, using Construction \ref{trunpoly}.
 Since $\K(-)$ is a lax symmetric monoidal functor \cite[Proposition 10.11]{BGT13}, and $\Mod_{\BP\langle n\rangle[x]/x^e}\to \Mod_{\BP\langle n\rangle\otimes_{\MU}\MU[x]/x^e}\to \Mod_{\BP\langle n\rangle}$ are $\Ee_2$-monoidal maps between $\Ee_2$-monoidal categories \cite[Proposition 4.5.8.20]{ha}. We therefore have following $\Ee_2$-ring maps \begin{equation*}
    \K(\BP\langle n\rangle[x]/x^e)\to \K(\BP\langle n\rangle\otimes_{\MU}\MU[x]/x^e)\to \K(\BP\langle n\rangle).
\end{equation*}
    Since $L_{K(n+1)}$ is a lax symmetrical monoidal functor, applying $L_{K(n+1)}$ yields following $\Ee_2$-ring maps
    \begin{equation*}
        L_{K(n+1)}(\K(\BP\langle n\rangle[x]/x^e))\to
        L_{K(n+1)}(\K(\BP\langle n\rangle\otimes_{\MU}\MU[x]/x^e))\to L_{K(n+1)}(\K(\BP\langle n\rangle)).
    \end{equation*}
 Suppose $L_{K(n+1)}(\K(\BP\langle n\rangle_{\MU}\otimes \MU[x]/x^e))\simeq 0$, then $L_{K(n+1)}(K(\BP\langle n\rangle))\simeq 0$. However, according to \cite[Corollary 5.0.2]{HW22}, we know that $L_{K(n+1)}\K(\BP\langle n\rangle)\not\simeq 0$, hence we arrive at a contradiction.  Thus $L_{K(n+1)}(\K(\BP\langle n\rangle\otimes_{\MU} \MU[x]/x^e))\not\simeq 0$, a same argument would arrive at the statement that \begin{equation*}
     L_{K(n+1)}(\K(\BP\langle n\rangle[x]/x^e))\not\simeq 0.
     \qedhere
 \end{equation*}
\end{proof}
We can also prove the following theorem with the same argument.
\begin{proposition}\label{redshift2}
    With the same notation as in Section \ref{sec2}, $\Lambda_{\BP\langle n\rangle}(x)$ is height $n$, $\BP\langle n\rangle/v_n^e$ is height $n-1$. And we have
\begin{equation*}
L_{K(n+1)}(\K(\Lambda_{\BP\langle n\rangle}))\not\simeq 0,\\
L_{K(n+1)}(\K(\BP\langle n\rangle\otimes_{\MU}\Lambda_{\MU}(x)))\not\simeq 0,\\
L_{K(n)}(\K(\BP\langle n\rangle/v_n^e))\not\simeq 0.
\end{equation*}
\end{proposition}

\begin{proof}
As seen in Proposition \ref{redshift1}, it suffice to produce $\Ee_3$-algebra maps \begin{equation*}
   \Lambda_{\BP\langle n\rangle}(x)\to \BP\langle n\rangle\otimes_{\MU} \Lambda_{\MU}(x)\to \BP\langle n\rangle, ~
\BP\langle n\rangle/v_n^e\to \BP\langle n-1\rangle.
\end{equation*}
By Remark \ref{remarke3alg2}, there is an $\Ee_3$-algebra map $\BP\langle n\rangle\otimes_{\MU} \Lambda_{\MU}(x)$ induced by an $\Ee_{\infty}$-ring map $\Lambda_{\Ss}(x)\to \Ss$, see Construction \ref{exterior}, \ref{trunpoly}. Then a similar argument would imply that $L_{K(n+1)}(\K(\BP\langle n\rangle\otimes_{\MU}\Lambda_{\MU}(x)))\not\simeq 0,     L_{K(n+1)}(\K(\Lambda_{\BP\langle n\rangle})(x))\not\simeq 0$.
For the second type ring spectra, we run induction on $e\in \Nn_{\geq 1}$. \begin{itemize}
    \item When $e=1$, by construction there is an equivalence $\BP\langle n\rangle/v_n\simeq \BP\langle n-1\rangle$, thus the desired statement is hold.
    \item Assume that $e=k$, we have an $\Ee_3$-algebra map $\BP\langle n\rangle/v_n^e\to \BP\langle n-1\rangle$.
    \item According to the property (2) in Proposition \ref{square0}, we have an $\Ee_3$-$\MU[y]/y^{k+1}$-algebra map $\BP\langle n\rangle/v_n^{k+1}\to \BP\langle n\rangle v_n^k$. Since the forgetful functor $\Mod_{\MU[y]/y^{k+1}}\to \Mod_{\Ss}$ is lax symmetric monoidal, therefore we get an $\Ee_3$-algebra map $\BP\langle n\rangle/v_n^{k+1}\to \BP\langle n\rangle/v_n^k$. 
    \item Composing the map right above with $\BP\langle n\rangle/v_n^k\to \BP\langle n-1\rangle$, yields an $\Ee_3$-algebra map $\BP\langle n\rangle/v_n^{k+1}\to \BP\langle n-1\rangle$.
\end{itemize}

Hence, for any $e\geq 1$, we conclude that $L_{K(n)}(\K(\BP\langle n\rangle/v_n^e))\not\simeq 0$ by the fact that $L_{K(n)}(\K(\BP\langle n-1\rangle))\not\simeq 0$.
\end{proof}
The results above actually has a refined version.
\begin{proposition}\label{NTC1}
The homotopy fixed point spectral sequence for $\THH(\BP\langle n\rangle[x]/x^e/\MU[x])^{h S^1}$ collapses at $E_2$-page.
 The map \begin{equation*}
    \pi_*\MU_{(p)}^{h S^1}\to \pi_*(\MU_{(p)}[x])^{h S^1} \to \pi_* \TC^-((\BP\langle n\rangle[x]/x^e)/\MU[x])
\end{equation*}
sends the complex orientation map to an element which is detected by $t$, and $v_{n+1}$ is send to $t(\sigma^2 v_{n+1})$. And $v_{n+1}$ is not nilpotent in $\TC^-_*((\BP\langle n\rangle[x]/x^e)/\MU[x])$.  
\end{proposition}
\begin{proof}
We have the homotopy limit spectral sequence derived by the skeleton filtration on $B S^1\simeq \Cc\Pp^{\infty}$, and the formula is given by \begin{equation*}
       E^2\simeq  \THH_*((\BP\langle n\rangle[x]/x^e)/\MU[x])[t]\Longrightarrow \TC^-_*((\BP\langle n\rangle[x]/x^e)/\MU[x]).
    \end{equation*}
Since every thing is concentrated in even degree by Proposition \ref{THHtruncated}, thus the spectral sequence collapses at $E_2$-page. Using the same argument in \cite[A.4..1, Remark 5.0.5]{HW22}, one shows that $v_{n+1}$ is detected in $\lim_{\Cc\Pp^1} \THH(\BP\langle n\rangle[x]/x^e/\MU[x])$.
There are $\Ee_1$-$\MU$-algebra maps 
\begin{equation*}
 \TC^-((\BP\langle n\rangle[x]/x^e)/\MU[x]) \to \TC^-((\BP\langle n\rangle[x]/x^e)/\MU[x]/x^e) \to \TC^-(\BP\langle n\rangle/\MU).
\end{equation*}
\begin{claim}
 The first map is an $\Ee_2$-$\MU^{h S^1}$-algebra map.   
\end{claim}
\begin{claimproof}
Recall that we have adjunction pair $f^*:\Mod_{\MU[x]/x^e}\leftrightarrows \Mod_{\MU[x]}:f_*$, see \cite[Proposition 7.1.2.7]{ha}.
Therefore we have lax symmetric monoidal natural transformations \begin{equation*}
    \THH  f_* \Longrightarrow f_* f^* \THH f_*  \simeq f_* \THH f^* f_* \Longrightarrow f_*\THH.
\end{equation*}
Here, the equivalence above follows from the facts that $f_*$ is symmetric monoidal and commutes with colimits. 
Therefore we have a functor in $\Fun_{\text{lax}}(\Alg_{\Ee_1}(\Mod_{\MU[x]/x^e})^{\otimes}, \Mod_{\MU[x]}^{\otimes})$, applying to $\BP\langle n\rangle[x]/x^e$ yields an $\Ee_2$-$\MU[x]$-algebra map.    
\end{claimproof}
\begin{claim}
The second $\Ee_1$-$\MU^{h S^1}$-algebra map above is induced by an $\Ee_1$-$\MU$-algebra map and applying $-^{h S^1}$.
\end{claim}
\begin{claimproof} we could construct an $\Ee_1$-$\MU$-algebra map
\begin{equation*}
    \THH((\BP\langle n\rangle[x]/x^e)/\MU[x]/x^e)\simeq \THH(\BP\langle n\rangle/\MU)\otimes_{\MU} \MU[x]/x^e\to \THH(\BP\langle n\rangle/\MU).
\end{equation*}
Here, the equivalence above is an equivalence of $\Ee_2$-$\MU[x]/x^e$-algebras, see the proof of Proposition \ref{THHexterior}, the $\Ee_1$-$\MU$-algebra map above is induced by Construction \ref{trunpoly}. 
\end{claimproof}
Finally, by \cite[Proposition 5.0.1]{HW22}, we learn that $v_{n+1}$ is not nilpotent in $\TC^-_*(\BP\langle n\rangle/\MU)$, hence $v_{n+1}$ is not nilpotent in
$\TC^-_*((\BP\langle n\rangle[x]/x^e)/\MU[x])$.
\end{proof}
\begin{proposition}\label{NTC2}
The homotopy fixed point spectral sequence of $\THH(\BP\langle n\rangle/v_n^e/\MU)^{h S^1}$ collapses at $E_2$-page.
  The map \begin{equation*}
    \pi_*\MU_{(p)}^{h S^1} \to \pi_* \TC^-((\BP\langle n\rangle)/v_n^e/\MU)
\end{equation*}
sends the complex orientation map to an element which is detected by $t$. Moreover, $v_n$ is not nilpotent in $\TC^-_*((\BP\langle n\rangle/v_n^e)/\MU)$.    
\end{proposition}
\begin{proof}
According to the proof of Proposition \ref{NTC1}, and Proposition \ref{THHquotient}, the homotopy fixed point spectral collapses at $E_2$-page.  And we have $\Ee_2$-$\MU^{h S^1}$-algebra map 
\begin{equation*}
\TC^-((\BP\langle n\rangle/v_n^e)/\MU) \to \TC^-(\BP\langle n-1\rangle/\MU)
\end{equation*}
induced by an $\Ee_3$-$\MU$-algebra map $\BP\langle n\rangle/v_n^e\to \BP\langle n-1\rangle$ constructed in the proof of Proposition \ref{redshift2}.  Since $v_n$ is not nilpotent in $\TC^-(\BP\langle n-1\rangle/\MU)$, hence $v_n$ is not nilpotent in $\TC^-(\BP\langle n\rangle/v_n^e/\MU)$.
\end{proof}
\begin{proposition}\label{NTC3}
  The map \begin{equation*}
    \pi_*\MU_{(p)}^{h S^1}\to \pi_*\Lambda_{\MU}(x)^{h S^1} \to \pi_* \TC^-(\Lambda_{\BP\langle n\rangle}(x)/\Lambda_{\MU}(x))
\end{equation*}
sends the complex orientation map to an element which is detected by $t$, and $v_{n+1}$ is send to $t(\sigma^2 v_{n+1})$. Moreover, $v_{n+1}$ is not nilpotent in $\TC^-(\Lambda_{\BP\langle n\rangle}(x)/\Lambda_{\MU}(x))$. 
\end{proposition}
\begin{proof}
The homotopy fixed point spectral sequence has the formula:
\begin{equation*}
    E^2\simeq \THH_*(\BP\langle n\rangle/\MU)[t]\otimes \Lambda(x)  \Longrightarrow \TC^-_*(\Lambda_{\BP\langle n\rangle}(x)/\Lambda_{\MU}(x)).
\end{equation*}
There is a ring map\begin{equation*}
    \TC^-(\Lambda_{\BP\langle n\rangle}(x))\to \TC^-(\BP\langle n\rangle/\MU)
\end{equation*}
induced by $\Ee_1$-$\MU^{h S^1}$-algebra map \begin{equation*}
    \THH(\Lambda_{\BP\langle n\rangle}(x)/\Lambda_{\MU}(x))\simeq\THH(\BP\langle n\rangle/\MU)\otimes_{\MU} \Lambda_{\MU}(x)\to \THH(\BP\langle n\rangle/\MU).
\end{equation*}
Since $v_{n+1}$ is not nilpotent in $\TC^-(\BP\langle n\rangle/\MU)$, so we conclude that $v_{n+1}$ is not nilpotent in $\TC^-(\Lambda_{\BP\langle n\rangle}(x))$.
\end{proof}
\section{Segal conjecture}\label{sec4}
We investigate Lichtenbaum\textendash Quillen phenomena in this section, we can reduce to proving the Segal conjecture and weak canonical vanishing theorem for certain cyclotomic spectra due to Theorem \ref{boundedTR}. In particular, we prove Proposition \ref{pro:polyBPNSegal}, \ref{pro:segalfails}. First we review some basic concepts and related techniques. 
\begin{definition}[{\cite[Proposition 3.2]{MR99}}]
We say a spectrum $X$ is $\pi$-finite if $X$ is bounded ($\exists a,b \in \Zz, X\in \Sp_{[a,b]}$), and $\pi_k X$ is a finite group for any $k\in \Zz$. A spectrum $X$ is called \textbf{fp-spectrum} if $X$ is bounded below, $p$-complete, and there exists a finite complex $F$ such that $F\otimes X$ is finite in the previous sense.
\end{definition}
\begin{remark}
The notation "fp" should be viewed as abbreviation of "finite presented", due to \cite[Proposition 3.2]{MR99}. Because a spectrum $X$ is an fp-spectrum, if and only if $X$ is $p$-complete and bounded below, and $\pi_*(\Ff_p\otimes X)$ is \textbf{finite presented} as $\calA_*$-comodule.
\end{remark}
For a $p$-local finite complex $X$, we can define type of $X$ using thick subcategory theorem. Then we can define fp-type for an fp-spectrum. 
\begin{definition}
We define \textbf{fp-type} for an fp-spectrum $X$ to be the minimum of the set $\{\text{type}(F)-1: F\otimes X ~\text{is  finite, and $F$ is finite complex}\}$.
\end{definition}
\begin{proposition}\label{type}
    The ring spectra $(\BP\langle n\rangle[x]/x^e)^{\wedge}_p, (\Lambda_{\BP\langle n\rangle}(x))^{\wedge}_p$ in Section \ref{sec2} are fp-spectra of type $n$, the ring spectra $(\BP\langle n\rangle/v_n^e)^{\wedge}_p$ is an fp-spectrum of type $n-1$.
\end{proposition}
\begin{proof}
The spectra mentioned above are both $p$-complete and bounded below. Since $\BP\langle n\rangle^{\wedge}_p$ is fp-spectrum of type $n$, now the result follows from Proposition \ref{finalgebra}, and note that there is equivalence $X^{\wedge}_p\otimes \Ss/p \simeq X\otimes \Ss/p$.
\end{proof}
Furthermore, we have following theorem establishing the relationship between fp-spectra and Lichtenbaum\textendash Quillen property.
\begin{theorem}[{\cite[Proposition 8.2]{MR99},\cite[Theorem 3.1.3]{HW22}\cite[Proposition 1.1]{yang2025}}]\label{thm:LQproperty}
Let $X$ be an fp-spectrum of type at most $n$, then the fiber of the map
\begin{equation*}
    X\to L^f_n X
\end{equation*}
is bounded above, where $L^f_n$ denote the Bousfield localization functor with respect to the spectrum $T(0)\oplus T(1)\oplus \cdots \oplus T(n)$, and $T(k)$ is the telescope of a type $k$ finite complex. In fact, if $X$ is a $p$-local spectrum and there is a type-$(n+1)$ finite spectrum $F$ such that $X\otimes F$ is bounded. Then the fiber of the localization map 
\begin{equation*}
    X\to L^f_n X
\end{equation*}
is bounded above.
\end{theorem}
We aim to investigate the redshift phenomena in this context, the methods to prove redshift conjecture were well understood, see \cite{AN21,HW22}.
More concretely, we have the following theorem. First, we introduce some conventions.
\begin{convention}
As in \cite{HW22}, for a connective algebra $A$, we will abbreviated $\TR^j(\THH(A)^{\wedge}_p))$ by $\TR^j(A)$, $\TR(\THH(A)^{\wedge}_p)$ by $\TR(A)$, and $\TC(\THH(A)^{\wedge}_p)$ by $\TC(A)$.
Let $X$ be a bounded below $p$-typical cyclotomic spectrum. We use following notations to indicate some statements concerning about $X$.
Let $F$ be a fixed finite complex.
\begin{itemize}
    \item \framebox{$F$-Bounded TR for $X$} = The spectrum $F\otimes \TR(X)$ is bounded.
    \item \framebox{$F$-Bounded TC for $X$} = The spectrum $F\otimes \TC(X)$ is bounded.
    \item \framebox{$F$-Segal Conjecture for $X$} = The fiber of the Frobenius $F\otimes X\to F\otimes X^{t C_p}$
    is bounded above.
    \item  \framebox{$F$-Weak Canonical Vanishing for $X$} = There is an integer $d\geq 0$, such that for any $i\geq d$ the map 
    \begin{equation*}
        \pi_i(F\otimes \mbox{can}) : \pi_i(F\otimes X^{h C_{p^k}})\to \pi_i(F\otimes X^{t C_{p^k}})
    \end{equation*}
    is zero for all $0\leq k\leq \infty$.
\end{itemize}
When the concerning object $X$ is clear, we will simply write as \framebox{$F$-Bounded $\TR$} without mentioning $X$. 
\end{convention}
\begin{theorem}[{\cite[Theorem 3.3.2]{HW22}}]\label{boundedTR}
Let $X$ be a bounded below, $p$-typical cyclotomic spectrum and $F\otimes X$ is $p$-power torsion ($\exists N, p^N=0 \in \pi_0(\Map(F\otimes X, F\otimes X))$). Then we have,
\begin{equation*}
\text{\framebox{$F$-\textup{Segal Conjecture}} $\wedge$
\framebox{$F$-\textup{Weak Canonical Vanishing}} 
$\Longleftrightarrow$
\framebox{$F$-\textup{Bounded} $\TR$}.}
\end{equation*}
\end{theorem}
\begin{remark}
Note that $\TC\simeq \fib(1-F: \TR\to  \TR)$, hence we have \framebox{$F$-Bounded $\TR$} $\Longrightarrow$ \framebox{$F$-Bounded $\TC$}. In fact, the $\TR$ here is the usual $\TR(-,p)$, and note that we have equivalences $\TR(A^{\wedge}_p)^{\wedge}_p\simeq \TR(A)^{\wedge}_p\simeq \TR(A,p)^{\wedge}_p$.
\end{remark}
Hence, according to the discussion above, to prove the boundedness of $\TC$, it suffice to prove Segal conjecture and weak canonical vanishing. However, for $\BP\langle -1\rangle [x]/x^e\simeq \Ff_p[x]/x^e$ the Segal conjecture fails, when $e>1$.
\begin{proposition}\label{pro:segal1}
Let $e>1$, the fiber of the Frobenius $$\THH(\Ff_p[x]/x^e)\otimes \Ss/p\to \THH(\Ff_p[x]/x^e)^{t C_p}\otimes \Ss/p$$ is not bounded above. Moreover, for any perfect field $k$ of characteristic $p$, the fiber of the Frobenius \begin{equation*}
 \THH(k[x]/x^e)\otimes \Ss/p\to \THH(k[x]/x^e)^{t C_p}\otimes \Ss/p  
\end{equation*} is not bounded above, i.e. 
\framebox{$\Ss/p$-\textup{Segal conjecture}}  fails for $\THH(k[x]/x^e)^{\wedge}_p$. Consequently, \framebox{$\Ss$-\textup{Segal conjecture}} also fails for both $\THH(k[x]/x^e)^{\wedge}_p$ and $\THH(k[x]/x^e)$. 
\end{proposition}
\begin{proof}
By \cite[Proposition 3.3.7]{HW22}, since $k[x]/x^e$ is a connective $\Ee_{\infty}$-$\BP\langle -1\rangle$-algebra, then for any type $1$ complex $F$, we have 
\begin{equation*}
\text{\framebox{$F$-Segal conjecture for $\THH(k[x]/x^e)^{\wedge}_p$} $\Longrightarrow$  \framebox{$F$-Bounded $\TR$ for $\THH(k[x]/x^e)^{\wedge}_p$}.}  
\end{equation*}
From the computation in \cite[Theorem 4.2.10]{HM97}, one gets that $\TC(k[x]/x^e)\otimes \Ss/p$ is not bounded above \footnote{Note for $k=\Ff_p$, the result says $\mathbf{W}_{me}(\Ff_p)/V_e(\mathbf{W}_m(\Ff_p)) \simeq \TC_{2m-1}(k[x]/x^e, (x))$, hence $\mathbf{W}_{me}(\Ff_p)/V_e(\mathbf{W}_m(\Ff_p)) \otimes \Zz/p \hookrightarrow  \TC_{2m-1}(k[x]/x^e; \Zz/p)$. Now note that $\mathbf{W}_{me}(\Ff_p)/V_e(\mathbf{W}_m(\Ff_p)) \otimes \Zz/p \twoheadrightarrow (\Zz/p^{me}/p^e (\Zz/p^m)) \otimes \Zz/p \Zz$, and $(\Zz/p^{me}/p^e (\Zz/p^m)) \otimes \Zz/p\Zz \not\simeq 0$ for $m\gg 0$, thus $\TC_{2m-1}(k[x]/x^e;\Zz/p)\not\simeq 0$. In general, combining the same argument with \cite[Theorem 2.1]{HMwitt}, one gets the desired statement.}, 
therefore $\TR(k[x]/x^e)\otimes \Ss/p$ is not bounded above. Hence, $\Ss/p$-Segal conjecture fails for $\THH(k[x]/x^e)^{\wedge}_p$.
\end{proof}

\begin{remark}
The proof of the proposition above may appear counterintuitive, as our goal is to apply Segal conjecture to deduce the boundedness of $\TR$, or $\TC$.
There is an alternative proof of the above result, and the result can be further generalized. That would require a better understanding of $\THH(\Ss[x]/x^e)$ which we will present later.
\end{remark}
First, we recall a framework used in \cite{HW22} to prove the Segal conjecture, and it also appeared in \cite{MaytypeSS}. 
Let $\fil(\Sp):=\Fun(\Zz^{\text{op}},\Sp)$ be the category of filtered spectra, which can be equipped with a symmetrical monoidal structure via Day convolution \cite[Example 2.2.6.17]{ha}. Then for an algebra $X$ in $\fil(\Sp)$, one can define $\THH(x)$ using \cite[Definition \Rom{3}.2.3.]{NS18}, and $\THH(X)$ is a filtered spectrum. And note that for an algebra $A$, we can easily viewed $A$ as filtered algebra via Whitehead tower  which we could denote it by $\tau_{\geq \bullet} A$. Then $\THH(\tau_{\geq \bullet} A)$ gives us a filtration on $\THH(A)$, and the associated spectral sequence is given by \begin{equation*}
    \pi_{*}\THH(\gr_{\bullet} A) \Longrightarrow \pi_* (\THH(A)).
\end{equation*}
There is also a well known way to produce a filtered algebra from an algebra, namely Adams descent. Let $E$ be a connective commutative algebra.
For a spectrum $X$, one can form then $E$-Adams tower (or $E$-synthetic analogue) $\lim_{[n]\in \triangle_{\leq \bullet}}X\otimes E^{\otimes[n]}$, then the associated spectral sequence is the Adams spectral sequence. There is a modified version, called descent tower in \cite{HW22} or $E$ synthetic analogue of $X$. The filtered spectrum is given by $ \lim_{[n]\in \triangle}( \tau_{\geq \bullet}( X\otimes E^{\otimes [n]}))$. We denote this functor by $\desc_E: \Sp \to \fil(\Sp)$ which is a lax symmetric monoidal functor,  and according to  \cite[Convention C.1.1]{HW22} we get a spectral sequence 
\begin{equation*}
    E^{s,t}_2\simeq \pi_{t-s}\gr^t X.
\end{equation*}
Here, we have used the \textit{d\'ecalage} construction, see \cite[Remark 3.7]{GIK22}. Especially, for $E=\Ff_p$, and $X$ is bounded below, we have the following spectral sequence \begin{equation*}
    E^{s,t}_2 \simeq \ext^{*,*}_{\calA}(\Ff_p,\coh_*(X;\Ff_p)) \Longrightarrow \pi_* X^{\wedge}_p.
\end{equation*}
Now, in \cite{HW22}, to prove the \framebox{$F$-Segal conjecture for $\THH(\BP\langle n\rangle)$} it suffice to show that the Frobenius map \begin{equation*}
    \Ff_p\otimes_{\BP\langle n\rangle} \THH(\BP\langle n\rangle) \to \Ff_p\otimes_{\BP\langle n\rangle} \THH(\BP\langle n\rangle)^{t C_p}
\end{equation*}
is truncated. Applying the descent functor $\desc$, one eventually reduce to show for the $\Ee_2$-graded algebra $\Ff_p \otimes \Ss[v_1] \otimes \cdots \otimes \Ss[v_n]$,  the Frobenius
\begin{equation*}
    \THH_*(\Ff_p\otimes \Ss[v_1]\otimes \cdots \otimes \Ss[v_n])/(v_1,v_2,\ldots, v_n) \to \THH_*(\Ff_p\otimes \Ss[v_1]\otimes \cdots \otimes \Ss[v_n])^{t C_p}/(v_1,v_2,\ldots, v_n)
\end{equation*}
is isomorphism in large enough degrees. Where the $\Ee_2$-algebra is equivalent to the associated graded algebra of $\desc_{\Ff_p}(\BP\langle n\rangle)$ by computing Adams $E_2$-page and some formal arguments. 
\begin{remark}
  Note that according to \cite[Proposition 6.6]{Lars-zetafunction}, the Segal conjecture is true for $\THH(\Ff_p[x])$ (and this can be also seen in \cite{HW22}). Hence it would make sense to guess that \framebox{$F$-Segal conjecture for $\THH(\BP\langle n\rangle[x])$} holds for any type $n+1$-finite complex $F$. And for $\BP\langle n\rangle[x]/x^e$, as seen in Proposition \ref{pro:segal1}, it would also make sense to guess that Segal conjecture fails for $\THH(\BP\langle n\rangle[x]/x^e)$. These all require some understanding of the associated graded algebras derived by $\Ff_p$-descent functor ($\Ff_p$-synthetic analogue).
\end{remark}
According to the discussion before, we first prove the following result. Consider the $\Ee_2$-graded-$\Ff_p[x]$-algebra $R$, with the homotopy ring $\pi_* R\simeq \Ff_p[x][a_1,a_2,\ldots ,a_n]$, each $a_i$ has degree $|a_i|\in 2\Zz$ and positive weight.
\begin{proposition}\label{pro:gradede2}
    The $\Ee_2$-graded-$\Ff_p[x]$-algebra $R$ is equivalent to 
    \begin{equation*}
        \Ff_p[x]\otimes \Ss[a_1]\otimes \Ss[a_2] \otimes \cdots\otimes \Ss[a_n],
    \end{equation*} with the corresponding degrees and weights, where $\Ss[a_i]$ is the free graded $\Ee_2$-algebra generated by $a_i$.
\end{proposition}
\begin{proof}
    The proof is the same as in \cite[Proposition 4.2.1]{HW22}, we denote $\Ff_p[x]\otimes \Ss[a_1]\otimes \Ss[a_2]\otimes \cdots \Ss[a_n]$ by $A$, we only need to show that 
    \begin{equation*}
        \Ff_p[x]\otimes_{\Ff_p[x]\otimes_{A} \Ff_p[x]} \Ff_p[x]
    \end{equation*}
has homotopy groups concentrated in even degrees. But this is just an application of \cite{Ang08} or Tor spectral sequence.
\end{proof}
In general, one can prove the following result.
\begin{proposition}\label{pro:gradedn}
Let $R$ be a commutative ring. For any graded $\Ee_2$-$R$-algebra $R(n)$ with the homotopy ring $\pi_*R(n) \simeq R[a_1,\ldots, a_n]$, we have an equivalence of graded $\Ee_2$-$R$-algebras
    \begin{equation*}
        R(n) \simeq R\otimes \Ss[a_1]\otimes \cdots\otimes \Ss[a_n].
    \end{equation*}
\end{proposition}
\begin{proposition}
Consider the $\Ee_2$-graded-$\Ff_p[x]$-algebra $R$, with the homotopy ring $\pi_* R\simeq \Ff_p[x][a_1,a_2,\ldots ,a_n]$, each $a_i$ has degree $|a_i|\in 2\Zz$ and positive weight.
    The map \begin{equation*}
        \pi_*(\THH(R))/(a_1,a_2,\ldots, a_n) \to \pi_*(\THH(R)^{t C_p})/(a_1, a_2, \ldots ,a_n)
    \end{equation*}
is an equivalence in degrees $*>n+1+\sum_{i=1}^n |a_i|$.
\end{proposition}
\begin{proof}
We use the argument in \cite[Corollary 4.2.3]{HW22}. By Proposition \ref{pro:gradede2}, we may assume that $R=\Ff_p[x]\otimes \Ss[a_1]\otimes \Ss[a_2]\otimes \cdots \Ss[a_n]$. Hence the Frobenius is equivalent to the composition of the map \begin{equation*}
    \THH(\Ff_p[x])\otimes \THH(\Ss[a_1]) \otimes \cdots \otimes \THH(\Ss[a_n]) \to \THH(\Ff_p[x])^{t C_p} \otimes \THH(\Ss[a_1])^{t C_p} \otimes \cdots \THH(\Ss[a_n])^{t C_p},
\end{equation*}
and the equivalence
\begin{equation*}
\THH(\Ff_p[x])^{t C_p}\otimes \THH(\Ss[a_1])^{t C_p} \otimes \cdots \otimes \THH(\Ss[a_n])^{t C_p} \to (\THH(\Ff_p[x])\otimes \THH(\Ss[a_1])\otimes \cdots\otimes \THH(\Ss[a_n]))^{t C_p}.
\end{equation*}
Note that the proof in \cite[Proposition 6.6]{Lars-zetafunction} shows that the Frobenius map $\varphi:\THH(\Ff_p[x]) \to \THH(\Ff_p[x])^{t C_p}$ is equivalent to the composition of map 
\begin{equation*}
    \THH(\Ff_p[x])\to \THH(\Ff_p)^{t C_p} \otimes \THH(\Ss[x]),
\end{equation*}
and the equivalence 
\begin{equation*}
    \THH(\Ff_p)^{t C_p} \otimes \THH(\Ss[x]) \to \THH(\Ff_p[x])^{t C_p}.
\end{equation*}
Thus we only need to show that the map 
\begin{equation*}
\THH(\Ff_p[x])\otimes \THH(\Ss[a_1, \ldots, a_n]) \to
    \THH(\Ff_p)^{t C_p}\otimes \THH(\Ss[x])\otimes \THH(\Ss[a_1, \ldots, a_n])^{t C_p}
\end{equation*}

has the desired property on homotopy groups. Note that we have following facts.
\begin{itemize}
    \item For any spectrum $X$, we have equivalence $X\otimes\Ff_p \simeq X^{\wedge}_p\otimes\Ff_p$.
    \item We have equivalence of $\Ee_{\infty}$-algebras 
    $ \THH(\Ff_p)\simeq \THH(\Ff_p)\otimes_{\Ff_p} \Ff_p$, and also $\THH(\Ff_p)^{t C_p} \simeq \THH(\Ff_p)^{t C_p}\otimes_{\Ff_p} \Ff_p$.
    \item $\Ff_p\otimes \THH(\Ss[x]) \simeq \THH(\Ff[x]/\Ff_p)$. According to \cite{Ang08}, we have the homotopy groups as $\pi_*(\Ff_p\otimes \THH(\Ss[x]))\simeq \Ff_p[x]\otimes_{\Ff_p}\Lambda_{\Ff_p}(\sigma x).$
\end{itemize}
Therefore we are reduce to understand the map
\begin{equation}\label{equ:pi*Fro}
\Ff_p[u]\otimes\Ff_p[x,a_1,\ldots,a_n]\otimes\Lambda_{\Ff_p}(\sigma x, \sigma a_1,\ldots ,\sigma a_n) \to \Ff_p[u^{\pm}]\otimes \Ff_p[x,a_1,\ldots,a_n]\otimes\Lambda_{\Ff_p}(\sigma x, \sigma a_1,\ldots ,\sigma a_n).
\end{equation}
 Modulo $a_1, a_2,\ldots, a_n$, the right hand side of (\ref{equ:pi*Fro}) may has elements  generated by 
 \begin{equation*}
   u^j (\sigma x)^{j_0} (\sigma a_1)^{j_1} \cdots (\sigma a_n)^{j_n}, 
 \end{equation*} which are not in the image of the left hand side, where $j>0, j_i \in \{0,1\}$.
If we require the total degree greater than $|\sigma x|+\sum_{i=1}^n |\sigma a_i|= n+1+\sum_{i=1}^n |a_i|$, such elements does not appear, hence the map modulo $a_1, a_2,\ldots, a_n$ is an equivalence for $*>|\sigma x|+\sum_{i=1}^n |\sigma a_i|= n+1+\sum_{i=1}^n |a_i|$. 
\end{proof}
This coincides with \cite[Proposition 6.6]{Lars-zetafunction}, more generally, one can prove the following theorem for $R=\Ff_p[x_1,\ldots,x_m]$ in Proposition \ref{pro:gradedn}.
\begin{proposition}
 The map induced by Frobenius
 \begin{equation*}
     \pi_*(\THH(R(n))/(a_1, a_2,\ldots, a_n) \to \pi_*(\THH(R(n))^{t C_p})/(a_1, a_2, \ldots ,a_n)
 \end{equation*}
is an equivalence in degree $*>n+m+\sum_{i=1}^n |a_i|$.
\end{proposition} 
\begin{proof}
    Same as before, and the statement in \cite[Proposition 6.6]{Lars-zetafunction} still works.
\end{proof}
With the above results, we could prove the Segal conjecture for $\BP\langle n\rangle[x]$ using the strategy mentioned before, more precisely, we have the following fact.
\begin{theorem}\label{pro:polyBPNSegal}
    The map of algebras induced by Frobenius \begin{equation*}
        \THH(\BP\langle n\rangle[x])\otimes_{\BP\langle n\rangle
        [x]}\Ff_p[x]\to \THH(\BP\langle n\rangle[x])^{t C_p}\otimes_{\BP\langle n\rangle[x]}\Ff_p[x]
    \end{equation*}
is an equivalence in large degrees, i.e. the map \begin{equation*}
    \THH(\BP\langle n\rangle[x])/(p,v_1, v_2, \ldots,v_n)\to \THH(\BP\langle n\rangle[x])^{t C_p}/(p, v_1, v_2,\ldots, v_n)
\end{equation*}
If one replaces $\BP\langle n\rangle$ with $\BP\langle n\rangle \otimes \Ss_{W(k)}$ the statement holds still. Here, $k$ is a perfect field of characteristic $p$, $\Ss_{W(k)}$ is the spherical Witt vectors.
\end{theorem}
\begin{proof}
The map \begin{equation*}
    \varphi: \THH(\BP\langle n\rangle[x])\to\THH(\BP\langle n\rangle[x])^{t C_p}
\end{equation*}
is a map of $\Ee_2$-algebra, and $\THH(\BP\langle n\rangle[x])$ is equipped with a $\BP\langle n\rangle[x]$-module structure, hence $\varphi$ can be promoted to a map of $\BP\langle n\rangle[x]$-modules, where the module structure of right hand side is induced by the $\Ee_2$-algebra map. 
Consider the filtered algebras $\desc_{\Ff_p}(\BP\langle n\rangle[x]), \desc_{\Ff_p}(\Ff_p)$. This will give a filtration on the Frobenius map \begin{equation*}
 \THH(\BP\langle n\rangle[x])\otimes_{\BP\langle n\rangle[x]}\Ff_p[x] \to \THH(\BP\langle n\rangle[x])^{t C_p}\otimes_{\BP\langle n\rangle[x]}\Ff_p.   
\end{equation*} According to \cite[Proposition C.5.4]{HW22}, $\THH(\desc_{\Ff_p}(\BP\langle n\rangle[x]))^{\wedge}_{v_0}$ converges to $\THH(\BP\langle n\rangle[x])^{\wedge}_p$, and $\THH(\desc_{\Ff_p}(\BP\langle n\rangle[x]))^{t C_p}$ converges to $\THH(\BP\langle n\rangle[x])^{t C_p}$, $\desc_{\Ff_p}(\BP\langle n\rangle[x])$ converges to $(\BP\langle n\rangle[x])^{\wedge}_p$.
Hence the map \begin{equation}\label{equ:filBPNfro}
    \THH(\desc_{\Ff_p}(\BP\langle n\rangle[x]))^{\wedge}_{v_0}\otimes_{\desc_{\Ff_p}\BP\langle n\rangle[x]} \Ff_p[x]\to \THH(\desc_{\Ff_p}\BP\langle n\rangle[x])^{t C_p}\otimes_{\desc_{\Ff_p}(\BP\langle n\rangle[x])} \Ff_p[x]
\end{equation}
gives arise to a filtration on \begin{equation}\label{equ:BPNfro}
    \THH(\BP\langle n\rangle[x])\otimes_{\BP\langle n\rangle[x]} \Ff_p[x] \to \THH(\BP\langle n\rangle[x])^{t C_p}\otimes_{\BP\langle n\rangle[x]} \Ff_p[x].
\end{equation}
Note that the associated graded has homotopy groups as \begin{equation*}
\pi_*(\gr_{\bullet}(\BP\langle n\rangle[x]))\simeq \ext^{*,*}_{\calA_*}(\Ff_p, \coh_*(\BP\langle n\rangle[x],\Ff_p)).
\end{equation*}
The $\Ff_p$-homolgy of $\BP\langle n\rangle$ is coextended from the quotient Hopf algebra $\Lambda({\overline{\tau_0},\overline{\tau_1},\ldots ,\overline{\tau_n}})$, hence $\coh_*(\BP\langle n\rangle[x],\Ff_p)\simeq \coh_*(\BP\langle n\rangle)\otimes_{\Ff_p} \coh_*(\Ss[x],\Ff_p)\simeq (\calA\square_{E_1} \Ff_p)\otimes (E_1\square_{E_1}\Ff_p[x])\simeq (\calA\square_{E_1} \Ff_p\otimes E_1)\square_{E_1}\Ff_p[x]\simeq \calA\square_{\Lambda(\overline{\tau_0},\ldots,\overline{\tau_n})}\Ff_p[x]$, where $E_1=\Lambda(\overline{\tau_0},\ldots,\overline{\tau_n})$.
Therefore, by change of ring isomorphism, we have \begin{equation*}
    \pi_*\gr_{\bullet} (\BP\langle n\rangle[x])\simeq \ext_{\Lambda(\overline{\tau_0},\ldots, \overline{\tau_n})}(\Ff_p,\Ff_p[x])\simeq \Ff_p[x,v_0,v_1,\ldots, v_n].
\end{equation*} 
Thus by Proposition \ref{pro:gradede2}, we have an equivalence of graded $\Ee_2$-algebras
$\gr_{\bullet}(\BP\langle n\rangle[x])\simeq \Ff_p[x]\otimes \Ss[v_0]\otimes \cdots\Ss[v_n]$.
Now according to Proposition \ref{pro:polyBPNSegal}, we have that 
the fiber of (\ref{equ:filBPNfro}) is bounded above, and hence the fiber of (\ref{equ:BPNfro}) is bounded above.
\end{proof}
\begin{corollary}\label{cor:polyBPNSegal}
Let $F$ be any finite complex of type $n+1$, then  
\begin{equation*}
\text{
\framebox{$F$-\textup{Segal conjecture for }$\THH(\BP\langle n\rangle[x])$}
} 
\end{equation*}
is true.
\end{corollary}
\begin{proof}
    Since the category of finite complex $F$ of type $n+1$ such that
\begin{equation*}
\text{\framebox{$F$-Segal conjecture for $\THH(\BP\langle n\rangle[x])$}}       
\end{equation*}  holds is thick, and the generalized Moore complex $\Ss/(p^{i_0}, v_1^{i_1}, \ldots, v_n^{i_n})$ generates the thick subcategory of type $n+1$ complex. It is enough to show that for any generalized Moore complex has that property. This follows from Proposition \ref{pro:polyBPNSegal}.
\end{proof}
Since $\BP\langle 1\rangle[x]$ is an $\Ee_{\infty}$-$\BP\langle 1\rangle$-algebra, see \cite{BR05}, hence according to \cite[Proposition 3.3.7]{HW22}, we have the following corollary.
\begin{corollary}\label{cor:LQforpoly}
Let $F$ be any $p$-local finite type $3$ complex. Then the following spectra are bounded.
\begin{align*}
&F\otimes \TC(\BP\langle 1\rangle[x]),\\
&F\otimes \K(\BP\langle 1\rangle[x])^{\wedge}_p.
\end{align*}
Consequently, the following maps have bounded above fibers
\begin{align*}
&\TC(\BP\langle 1\rangle[x])\to L_2^f (\TC(\BP\langle 1\rangle[x])),\\
&\K(\BP\langle 1\rangle[x])^{\wedge}_p \to L_2^f(\K(\BP\langle 1\rangle[x])^{\wedge}_p).
\end{align*}
\end{corollary}
\begin{proof}
The statement for $\TC$ follows from Corollary \ref{cor:polyBPNSegal}, \cite[Proposition 3.3.7]{HW22}, Theorem \ref{thm:LQproperty}. Note that we have a pull back square \begin{equation*}
\xymatrix{
\K(\BP\langle 1\rangle[x])^{\wedge}_p \ar[r]\ar[d]& \TC(\BP\langle 1\rangle[x])\ar[d]\\
\K(\Zz_{(p)}[x])^{\wedge}_p \ar[r]& \TC(\Zz_{(p)}[x])
}
\end{equation*}
Therefore, after applying $F\otimes-$, the above square is still a pull back square.
Since $\Zz_{(p)}[x]$ is an $\Ee_{\infty}$-$\BP\langle 0\rangle$-algebra, hence for any $F'$ a $p$-local type $2$ finite complex, $F'\otimes \TC(\Zz_{(p)}[x])$ is bounded. Then $F\otimes \TC(\Zz_{(p)}[x])$ is also bounded. 
Note that K-theory is $\Aa^1$-invariant for regular rings, then we have 
$\K(\Zz_{(p)}[x])\simeq \K(\Zz_{(p)})$. Because $\K(\Zz_{(p)})$ is an fp spectrum of type $1$, then $F\otimes \K(\Zz_{(p)}[x])$ is bounded. Therefore $F\otimes \K(\BP\langle 1\rangle[x])^{\wedge}_p$ is bounded.

The last statement follows from Theorem \ref{thm:LQproperty}.
\end{proof}
We will use the decomposition of $\THH(\Ss[x]/x^e)$ to study the Frobenius map of $\THH(\BP\langle n\rangle[x]/x^e)$. We recall the following decomposition in \cite[Theorem B]{HM97}.
\begin{proposition}\label{pro:decomposation}
Let $V_d=\Cc(\xi_m)\oplus \cdots\oplus \Cc(\xi_m^d)$, where $\xi_m$ is the $m$th primitive root of unity. We use $S^{V_d}$ to denote the one-point compactification of $V_d$.
Then $\THH(\Ss[x]/x^e)\simeq \bigoplus_{m\geq 0} B_m$, as $S^1$-spectra, where 
\begin{equation*}
B_m\simeq S^{V_{[\frac{m-1}{e}]}} \otimes \suspen_+ S^1/C_m,
\end{equation*}
if $e\nmid m$, and
\begin{equation*}
    B_m\simeq S^{V_{[\frac{m-1}{e}]}}\otimes \cofib(\suspen_+ S^1/C_{m/e} \to \suspen_+ S^1/C_m)
\end{equation*}
if $e\mid m$.
\end{proposition}
Now we could prove the following proposition, according to the above discussion.
\begin{proposition}\label{pro:segalfails}
Let $F$ be any $p$-local type $n+2$ finite complex. Let $W$ be the fiber of Frobenius map of $F\otimes\THH(\BP\langle n\rangle)$. The following statements are true.
\begin{enumerate}
    \item If $W\not\simeq 0$, then the fiber of the Frobenius \begin{equation*}
    F\otimes \THH(\BP\langle n\rangle[x]/x^e) \to F\otimes \THH(\BP\langle n\rangle[x]/x^e)^{t C_p}
\end{equation*}
is not bounded above, where $e>1$.  In other words, $F$-Segal conjecture for $\THH(\BP\langle n\rangle[x]/x^e)$ is true if and only if $W\simeq 0$.
\item If $F$ is a sum of generalized Moore complexes $\Ss/(p^{i_0}, v_1^{i_1}, \ldots v_{n+1}^{i_{n+1}})$, then $W\not\simeq 0$.
\item In fact, for $F$ a $p$-local type $n+2$ finite complex, $W\not\simeq 0$.
\end{enumerate}
\end{proposition}
\begin{lemma}
The Frobenius map of $\THH(\BP\langle n\rangle[x]/x^e)$ is the product of maps
\begin{equation*}
     \THH(\BP\langle n\rangle)\otimes B_{m/p} \to  \THH(\BP\langle n\rangle)^{t C_p} \otimes B_{m/p}\to (\THH(\BP\langle n\rangle)\otimes B_m)^{t C_p}.
\end{equation*}
\footnote{As one can see, one can replace $\BP\langle n\rangle$ by any $\Ee_1$-ring $R$ to get a more general statement for $\THH(R[x]/x^e)$.}
\label{unstablefro}
The second map is called unstable Frobenius which is an equivalence \begin{equation*}
     \THH(\BP\langle n\rangle)^{t C_p}\otimes B_{m/p} \to  \THH(\BP\langle n\rangle)\otimes B_m)^{t C_p}.
    \end{equation*}
\end{lemma}
\begin{proof}
    Note that the Frobenius map can be induced by the genuine cyclotomic structure, and the two Frobenius maps are equivalent. We have the following commutative diagram
    \begin{equation*}
    \xymatrix{
     \THH(\BP\langle n\rangle) \otimes B_{m/p} \ar[r]^{\simeq\quad}\ar[d]^{\simeq}& \Phi^{C_p}(\THH(\BP\langle n\rangle))\otimes B_{m/p}\ar[d]\\
     \Phi^{C_p} (\THH(\BP\langle n\rangle))\otimes B_m^{C_p} \ar[r]\ar[d]^{\simeq} & \THH(\BP\langle n\rangle)^{t C_p} \otimes B_m^{C_p}\ar[d]^{\simeq}\\
     \Phi^{C_p} (\THH(\BP\langle n\rangle)\otimes B_m) \ar[r] & (\THH(\BP\langle n\rangle\otimes B_m))^{t C_p}
     }
    \end{equation*}
The left vertical maps are equivalences, which are induced by the genuine cyclotomic structure on $\THH(\BP\langle n\rangle)$. And the top horizontal map is also an equivalence induced by the genuine cyclotomic structure.
The right vertical map is an equivalence by \cite[Lemma 9.1]{HMwitt}. Finally, we have $B_m^{C_p} \simeq  B_{m/p}$ where both sides are $S^1$-spaces, by \cite[(4.2.2)]{HM97}. Therefore the Frobenius map is equivalent to \begin{equation*}
    \THH(\BP\langle n\rangle)\otimes B_{m/p} \to  \THH(\BP\langle n\rangle)^{t C_p} \otimes B_{m/p}\to (\THH(\BP\langle n\rangle)\otimes B_m)^{t C_p},
\end{equation*}
where the second map is an equivalence.
\end{proof}
\begin{proof}[Proof of Proposition \ref{pro:segalfails}]
We have the following equivalences, 
\begin{equation*}
 F\otimes \THH(\BP\langle n\rangle[x]/x^e) \simeq F\otimes \THH(\BP\langle n\rangle) \otimes\THH(\Ss[x]/x^e)\simeq \bigoplus_{m\geq 0} F\otimes \THH(\BP\langle n\rangle)\otimes B_m, 
\end{equation*} where $B_m$ was described in Proposition \ref{pro:decomposation}, and the connectivity of $B_m$ tends to $\infty$, when $m\to \infty$. Hence we have the following equivalences (note that the summand has connectivity tends to $\infty$, then the direct sum is equivalent to the product).
\begin{equation*}
    F\otimes \THH(\BP\langle n\rangle[x]/x^e)\simeq \bigoplus_{m\geq 0} F\otimes \THH(\BP\langle n\rangle)\otimes B_m \simeq \prod_{m\geq 0} F\otimes \THH(\BP\langle n\rangle)\otimes B_m. 
\end{equation*}
Thus the Frobenius map factors as following,
\begin{equation*}
    \xymatrix{
    F\otimes \THH(\BP\langle n\rangle[x]/x^e) \ar[r]^{\varphi}\ar[d]_{\simeq} & F\otimes \THH(\BP\langle n\rangle[x]/x^e)^{t C_p}\ar[d]_{\simeq}\\
    \prod_{m\geq 0} F\otimes \THH(\BP\langle n\rangle)\otimes B_m \ar[r]^{\varphi'} \ar[d]_{\id\otimes\varphi\otimes\id} &  F\otimes\prod_{m\geq 0} (\THH(\BP\langle n\rangle)\otimes B_m)^{t C_p} \ar[d]_{\simeq }\\
    \prod_{m\geq 0} F\otimes \THH(\BP\langle n\rangle)^{t C_p} \otimes B_m \ar[r]^{\iota}  &  \prod_{m\geq 0} F\otimes (\THH(\BP\langle n\rangle) \otimes B_m)^{t C_p}.\\
    }
\end{equation*}
The right lower vertical map is an equivalence. Because $F$ is finite and then dualizable in $\Sp$ (see for instance \cite[Theorem 5.2.1]{orangebook}), and hence $F\otimes -$ commutes with products.
From the above diagram, it suffice to understand the map $\varphi'$. Suppose $\varphi'$ induce isomorphism on homotopy group on large degrees, then we have the component map \begin{equation*}
\xymatrix{F\otimes \THH(\BP\langle n\rangle) \otimes B_m \ar[rr]^{\id\otimes \varphi\otimes \id} &&  F\otimes \THH(\BP\langle n\rangle)^{t C_p} \otimes B_m \ar[r]^{\simeq}& F\otimes (\THH(\BP\langle n\rangle \otimes B_{mp})^{t C_p}}
\end{equation*}
induce isomorphism in large degrees, and the equivalence is due to Lemma \ref{unstablefro}.
\begin{claim}
We claim that $\exists b, \pi_b W\otimes \suspen_+ S^1\ne 0$ for some $b\in \Zz$. 
\end{claim}
\begin{claimproof}
Otherwise, $W\otimes \suspen_+ S^1 \simeq 0$.  Actually, one computes $\pi_*(\suspen S^1_+) \simeq \pi_*(\suspen S^1) \oplus \pi_*(\Ss)$, hence the Tor spectral sequence for $W\otimes \suspen_+ S^1$ has a summand $\pi_*(W) \not\simeq 0$, a contradiction.
\end{claimproof}
Now, we know that $\pi_{b+[\frac{m-1}{e}]} W\otimes B_m\ne 0$, for $e\nmid m$. When $m \to \infty$, this would give a  sequence of nonzero homotopy classes, a contradiction.

Next, we show that $W\not\simeq 0$. To this end, we show that $F\otimes \THH(\BP\langle n\rangle)^{t C_p}$ is not connective, and in fact $\pi_{-N} (F\otimes \THH(\BP\langle n\rangle))\not\simeq 0$, for suitable $N$ and $N\gg 0$. First let $F$ be a  generalized Moore complex. By using $\desc_{\Ff_p}$,  we learn that $\THH(\BP\langle n\rangle)^{t C_p}$ has associated graded equivalent to $\THH(\Ff_p[v_0, v_1, \ldots, v_n])^{t C_p}$. And it has homotopy groups as  \begin{equation*}
    \Ff_p[u^{\pm }, v_0, \ldots, v_n]\otimes \Lambda_{\Ff_p} (\sigma v_0, \ldots, \sigma v_n).
\end{equation*}
See \cite[Proposition 4.2.2]{HW22}.
We know that $v_0, \ldots, v_n$ is detected by the same elements above. Then $\THH(\desc_{\Ff_p
}(\BP\langle n\rangle))^{t C_p}\otimes_{\desc_{\Ff_p}\BP\langle n\rangle} \desc_{\Ff_p}\BP\langle n\rangle/(v_0^{i_0}, \ldots, v_n^{i_n})$ will have graded as \begin{equation}\label{equ:modulograded}
    \Ff_p[u^{\pm},v_0,\ldots, v_n]/(v_0^{i_0}, \ldots, v_n^{i_n})\otimes\Lambda (\sigma v_0, \ldots, \sigma v_n),
\end{equation}
Note that $u$ has degree $-2$, weight $0$, $|v_{n+1}^{i_{n+1}}|=(2p^{n+1}-2)^{i_{n+1}}$, and it lies in weight $(2p^{n+1}-1)^{i_{n+1}}$. The element $v_{n+1}^{i_{n+1}}$ induce a graded map $\Sigma^{|v_{n+1}^{i_{n+1}}|}\gr(w(v_{n+1}^{i_{n+1}})) \to \gr$.
If modulo $v_{n+1}$, $\THH(\BP\langle n\rangle)^{t C_p}/(p^{i_0},\ldots, v_n^{i_n})$ is connective, then 
the graded map induce isomorphism on $E_{\infty}$-page at negative degrees. In particular,  $
E_{\infty}^{-w(v_{n+1}^{i_{n+1}}), -2N-|v_{n+1}^{n+1}|-w(v_{n+1}^{i_{n+1}})} \simeq E_{\infty}^{0, -2N}$, however, when $N\to \infty$, the left hand side is zero (because there is no negative weight elements), while the right hand side is generated by $u^{-N}$. This follows from the following claim.
\begin{claim}
For $N\gg 0$, the element $u^{p^N}$ is a permanent cycle.
\end{claim}
\begin{claimproof}
Note that there is a horizontal vanishing line for the spectral sequence. So using Leibniz rule, one shows that for $N\gg 0$, the element $u^{p^N}$ is a permanent cycle. 
\end{claimproof}
A contradiction, meaning that $\THH(\BP\langle n\rangle)^{t C_p}/(p^{i_0},v_0^{i_1},\ldots,v_{n+1}^{i_{n+1}})$ is not connective. Furthermore, $\pi_{-2p^N} (F\otimes\THH(\BP\langle n\rangle)^{t C_p})\not\simeq 0$, for $N\gg 0$.
Then one sees that the fiber of Frobenius map cannot be an equivalence, hence $W\not\simeq 0$. In general, since $u^{p^N}$ lives in $E^{0, -2p^N}_{\infty}$, then by a thick sub-category argument, one proves that $W\not\simeq 0$
\end{proof}
As a corollary, we have following.
\begin{corollary}\label{cor:TRboundedfails}
Let $F$ be any $p$-local type $n+2$ finite complex. 
Then $F\otimes \TR(\BP\langle n\rangle[x]/x^e)$ is not bounded.
\end{corollary}
\begin{proof}
    This follows from Proposition \ref{pro:segalfails} and Theorem \ref{boundedTR}.
\end{proof}

\begin{theorem}\label{thm:generaltruncatedpoly}
 Let $S$ be a fp ring spectrum of type $n$, $F$ be a $p$-local finite $n+2$ finite complex. Let $V$ be the fiber of the Frobenius of $F\otimes \THH(S)$. Then $F\otimes \TR(S[x]/x^e)$ is bounded if and only if $V \simeq 0$. Here $e>0$.
\end{theorem}
\begin{proof}
    Same argument as in Theorem \ref{thm:generaltruncatedpoly}.
\end{proof}

Using the argument before we can prove the following.
\begin{proposition}\label{pro:segalforalge}
    The fiber of the Frobenius map  \begin{equation*}
        \Ff_p\otimes_{\BP\langle n\rangle} \THH(\BP\langle n\rangle[x]/x^e) \to \Ff_p\otimes_{\BP\langle n\rangle} \THH(\BP\langle n\rangle[x]/x^e)^{t C_p}
    \end{equation*}
is not bounded above.
\end{proposition}
As usual, we prove the graded version of this.
\begin{proposition}\label{pro:segalfailsgradtruncated}
    Let $A$ be the $\Ee_2$-graded algebra $\Ff_p[x]/x^e\otimes \Ss[a_0]\otimes\Ss[a_1]\otimes \Ss[a_n]$. Where $|a_i|\in 2\Zz$, and in some (positive) weights. Modulo $a_0,\ldots, a_n$ the fiber of the Frobenius map
    \begin{equation*}
        \THH(A)/(a_0,\ldots,a_n)\to\THH(A)^{t C_p}/(a_0, \ldots,a_n)
    \end{equation*}
is not bounded above. 
\end{proposition}
\begin{proof}
Same as before. The Frobenius map for $\THH(A)$ is equivalent to \begin{equation*}
    \THH(\Ff_p[x]/x^e)\otimes \THH(\Ss[a_0])\otimes\cdots \THH(\Ss[a_n]) \to \THH(\Ff_p[x]/x^e)^{t C_p}\otimes \THH(\Ss[a_0])^{\wedge}_p\otimes\cdots \THH(\Ss[a_n])^{\wedge}_p.
\end{equation*}
And hence equivalent to \begin{equation*}
    \THH(\Ff_p[x]/x^e)\otimes\THH(\Ff_p[a_0,\ldots,a_n]/\Ff_p)\to \THH(\Ff_p[x]/x^e)\otimes \THH(\Ff_p[a_0,\ldots, a_n]/\Ff_p).
\end{equation*}
Let $W$ be the fiber of $\varphi_p: \THH(\Ff_p[x]/x^e)\to \THH(\Ff_p[x]/x^e)^{t C_p}$, then we would like to understand $W\otimes \THH(\Ff_p[a_0,\ldots,a_n]/\Ff_p)$.
Note that $\THH(\Ss[x]/x^e)\simeq \bigoplus_{m\geq 0} B_m$, where $B_m$ is stated in Proposition \ref{pro:decomposation}. And hence the when $e\nmid m,$ the unstable Frobenius map is equivalent to \begin{equation*}
    \THH(\Ff_p)\otimes B_{m/p} \to \THH(\Ff_p)^{t C_p}\otimes B_{m/p}\to (\THH(\Ff_p)\otimes B_m)^{t C_p}
\end{equation*}
where the second map is an equivalence. Let $W_m$ be the fiber of the first map, then $W\simeq \prod_{m\geq 0} W_m$.
Note that when $m\to \infty$, $e\nmid m$, we have that $\pi_{[m-1/e]-2} W_m\not\simeq 0$, \footnote{In fact, by Proposition \ref{pro:segal1}, we know that there is a sequnece, $\{a_m\}$, with $\lim_{m\to \infty} a_m=\infty$, such that $\pi_{a_m}(W)\not\simeq 0$.} by the fact that the fiber of Frobenius map of $\THH(\Ff_p)$ is $\tau_{<0} \THH(\Ff_p)^{t C_p}$. And $\THH_*(\Ff_p[a_0,\ldots,a_n]/\Ff_p)/(a_0,\ldots,a_n)\simeq \Lambda_{\Ff_p}(\sigma a_0, \ldots, \sigma a_n)$. Hence $\pi_{[\frac{m-1}{e}]+|a_0|-1} (W\otimes_{\Ff_p} \HH(\Ff_p[a_0,\ldots,a_n]/\Ff_p)/(a_0,\ldots,a_n)) \not\simeq 0$. When $m$ tends to $\infty$, we get the desired statement.
\end{proof}
\begin{proof}[Proof of Proposition \ref{pro:segalforalge}]
Let the fiber of the map in the statement be $W_1$.
Since $\THH(\BP\langle n\rangle[x]/x^e)$ can be viewed as an $\Ee_1$-$\BP\langle n\rangle$-algebra, and we have $\Ee_2$-algebra map \begin{equation*}
    \THH(\BP\langle n\rangle[x]/x^e) \to \THH(\BP\langle n\rangle[x]/x^e)^{t C_p}.
\end{equation*}
Then the target of the map can be promoted to a $\BP\langle n\rangle$-module, therefore the map is well defined. We then apply $\desc_{\Ff_p}$ to the algebras above. Thus the map is approximate by \begin{equation*}
    \Ff_p\otimes_{\desc_{\Ff_p}(\BP\langle n\rangle)} \THH(\desc_{\Ff_p}(\BP\langle n\rangle[x]/x^e))\to \Ff_p\otimes_{\desc_{\Ff_p}(\BP\langle n\rangle)} \THH(\desc_{\Ff_p}(\BP\langle n\rangle[x]/x^e))^{t C_p}.
\end{equation*}
Hence the fiber is approximate by the fiber of this map of filtered spectra. The associated graded $\pi_* \gr \desc_{\Ff_p}(\BP\langle n\rangle) \simeq \Ff_p[x]/x^e[v_0,\ldots, v_n]$, then by Proposition \ref{pro:gradedn}, \begin{equation*}
\gr \desc_{\Ff_p} (\BP\langle n\rangle[x]/x^e)\simeq \Ff_p[x]/x^e \otimes \Ss[v_0]\otimes\cdots \otimes\Ss[v_n]=A. 
\end{equation*}
Then $W_1$ has graded given by the fiber of the map \begin{equation*}
    \THH(A)/(v_0,\ldots,v_n) \to \THH(A)^{t C_p}/(v_0,\ldots, v_n).
\end{equation*}
Taking homotopy groups, the map is \begin{equation*}
    \Ff_p[u]\otimes \HH_*(\Ff_p[x]/x^e/\Ff_p)\otimes \Lambda(\sigma v_1 ,\ldots ,\sigma v_n) \to \Ff_p[u^{\pm}] \otimes\HH_*(\Ff_p[x]/x^e/\Ff_p) \otimes \Lambda(\sigma v_1, \ldots \sigma v_n).
\end{equation*}\footnote{This follows from the fact that the Frobenius map $\THH(\Ss[x]/x^e)\to \THH(\Ss[x]/x^e)^{t C_p}$ is a $p$-adic equivalence.}
Note that $\HH_{2j}(\Ff_p[x]/x^e/\Ff_p)\simeq \Ff_p[x]/x^e$, let $\alpha \in \HH_2(\Ff_p[x]/x^e/\Ff_p)$. A similar argument shows that $\alpha^{p^N}$ is a permanent cycle, for $N\gg 0$. However $u^{p^N}\alpha^{p^{N+m}}$ is not in the image of the Frobenius map. When $m\to \infty$, this gives a sequence of homotopy classes in the fiber of the Frobenius map, hence the fiber of the Frobenius map is not bounded above.
\end{proof}

Actually, the statement applied in a general setting.
\begin{proposition}\label{pro:segalfailsforgroupring}
    Let $Y$ be a pointed simply connected space. By \cite[Corollary \Rom{4}.3.3]{NS18}, we know that $\THH(\Ss[\Omega Y]) \simeq \suspen_+ \mathcal{L} Y$. Suppose we have a decomposition \begin{equation*}
        \suspen_+\mathcal{L} Y\simeq \bigoplus_{m\geq 0} B_m
    \end{equation*}
as $S^1$-spectra. Let $p>2$, and  $F$ be any type $n+2$  $p$-local generalized Moore complex, i.e. $F=\Ss/(p, v_1^{k_1}, \ldots, v_{n+1}^{k_{n+1}})$. And assume further
\begin{enumerate}
    \item  When $m\to \infty$, the connectivity of $B_m$ tends to $\infty$. 
    \item The Frobenius of is given by product of $B_m \to B_{mp}^{t C_p}$, and $(B_m)^{t C_p}\simeq 0$ when $p\nmid m$.
    \item For each $B_m$ as a $C_p$-spectrum, lies in the thick subcategory that generated by $\Ss$ and $\suspen_+C_p$.
\end{enumerate}
Then  the Frobenius map of $F\otimes \THH(\BP\langle n\rangle[\Omega Y])$ is equivalent to the product of maps \begin{equation*}
    F\otimes \THH(\BP\langle n\rangle)\otimes B_m \to F\otimes \THH(\BP\langle n\rangle)^{t C_p} \otimes B_{mp}^{t C_p} \to F\otimes (\THH(\BP\langle n\rangle) \otimes B_{mp})^{t C_p} 
\end{equation*}
where the the second map is an equivalence, and the first map is a $p$-adic equivalence.
If further more, there are sequences $\{a_k\}, \{f(a_k)\}\subseteq \Zz$, such that $\pi_{f(a_k)} \Ss/p\otimes W\otimes B_{a_k}\not\simeq 0, \lim_{k\to\infty} f(a_k) =\infty$, where $W$ is the fiber of the Frobenius map of $F\otimes\THH(\BP\langle n\rangle)$, then
\begin{equation*}
\text{
\framebox{$F$-\textup{Segal conjecture for }$\THH(\BP\langle n\rangle[\Omega Y])$}
} 
\end{equation*}
fails.
As a by-product, $F\otimes \TR(\BP\langle n\rangle[\Omega Y])$ is not bounded.
\end{proposition}
\begin{proof}
    The argument is same as before. First we have the following equivalences by (1).\begin{equation*}
        F\otimes \THH(\BP\langle n\rangle[\Omega Y]) \simeq \bigoplus_{m\geq 0} F\otimes \THH(\BP\langle n\rangle)\otimes B_m \simeq \prod_{m\geq 0} F\otimes \THH(\BP\langle n\rangle) \otimes B_m.
    \end{equation*} 
And then by condition (1), (2), we have the following commutative diagram.
Hence the Frobenius map factors as follows,
\begin{equation*}
    \xymatrix{
    F\otimes \THH(\BP\langle n\rangle[\Omega Y]) \ar[r]^{\varphi}\ar[d]_{\simeq} & F\otimes \THH(\BP\langle n\rangle[\Omega Y])^{t C_p}\ar[d]_{\simeq}\\
    \prod_{m\geq 0} F\otimes \THH(\BP\langle n\rangle)\otimes B_{m/p}\ar[r]^{\varphi'} \ar[d]_{\id\otimes\varphi\otimes\varphi_m} &  F\otimes\prod_{m\geq 0} (\THH(\BP\langle n\rangle)\otimes B_m)^{t C_p} \ar[d]_{\simeq }\\
    \prod_{m\geq 0} F\otimes \THH(\BP\langle n\rangle)^{t C_p} \otimes B_m^{t C_p} \ar[r]^{\iota}  &  \prod_{m\geq 0} F\otimes (\THH(\BP\langle n\rangle) \otimes B_m)^{t C_p}.
    }
\end{equation*}
The right lower vertical map is an equivalence. Because $F$ is finite then dualizable is $\Sp$, thus commutes with products.
Therefore, the total Frobenius map is a product of maps
\begin{equation}\label{map:frobenius}
    F\otimes \THH(\BP\langle n\rangle)\otimes B_{m}\to F\otimes \THH(\BP\langle n\rangle)^{t C_p}\otimes B_{mp}^{t C_p}\to  F\otimes (\THH(\BP\langle n\rangle) \otimes B_{mp})^{t C_p},
\end{equation}
where the second map is an equivalence by condition (3) and combining with the argument in \cite[Proof of Proposition 4.2.2]{HW22}.
Now we are reduce to understand the first map in (\ref{map:frobenius}), we make the following claim.
\begin{claim}\label{claim:frobeniusmap}
The first map in (\ref{map:frobenius}) is $p$-adically equivalent to \begin{equation*}
    F\otimes \THH(\BP\langle n\rangle) \otimes B_m \to F\otimes \THH(\BP\langle n\rangle)^{t C_p} \otimes B_m
\end{equation*}
\end{claim}
\begin{claimproof}
Note that according to \cite[Theorem \Rom{4}.3.7]{NS18}, we know that the Frobenius map \begin{equation*}
    \THH(\Ss[\Omega Y])\to \THH(\Ss[\Omega Y])^{t C_p}
\end{equation*}
is a $p$-adic equivalence, and then by smashing $\Ss/p$ the map is also an equivalence. Since $\Ss/p$ is compact, then $\Ss/p\otimes -$ commutes with limits. Hence \begin{equation*}
    \Ss/p\otimes \varphi_m: \Ss/p\otimes B_{m/p} \to \Ss/p\otimes B_m^{t C_p} 
\end{equation*}
is a equivalence. And note the after smashing $\Ss/p$ with the first map in (\ref{map:frobenius}) the resulting map is equivalent to the original map, because $p>2$. Then we arrive at the desired statement.
\end{claimproof}
Suppose $F$-Segal conjecture for $\THH(\BP\langle n\rangle[\Omega Y])$ is true. By the previous discussion, the map in Claim \ref{claim:frobeniusmap} has bounded above fiber. However, this would cause a contradiction by condition (4). The last statement now follows from Theorem \ref{boundedTR}.
\end{proof}
\begin{corollary}\label{cor:segalfailsforgroupring}
Let $F$ be any $p$-local type $n+2$ finite complex. Let $W$ be the fiber of $\varphi: F\otimes\THH(\BP\langle n\rangle)\to F\otimes \THH(\BP\langle n\rangle)^{t C_p}$.
Then the following statements are true.
\begin{enumerate}
\item If $\Ss/p \otimes W\not\simeq 0$, then 
$F\otimes \TR(\BP\langle n\rangle[\Omega S^3])$ is not bounded, and also the $F$-Segal conjecture fails for $\THH(\BP\langle n\rangle[\Omega S^3])$. 
\item  We have $\Ff_p\otimes W\neq 0$. In particular, $\Ss/p\otimes W\ne 0$.
\end{enumerate}
\end{corollary}
\begin{proof}
By \cite[\S 5, Lemma 4.9, Lemma 5.5]{kthofthh}, $\THH(\Ss[\Omega S^3])\simeq \bigoplus_{m\geq 0} B_m$, and then condition (1), (2), (3) in Proposition \ref{pro:segalfailsforgroupring} hold.
\footnote{Actually, note that as $\Ee_1$-algebras $\Ss[\Omega S^3]$ is equivalent to $\oplus_{k\geq 0} (S^2)^{\otimes k}$. By \cite[Theorem 3.8]{AM21}, we know that $\THH(\Ss[\Omega S^3]) \simeq \bigoplus_{m\geq 0} \textup{Ind}^{S^1}_{C_m} (S^2)^{\otimes m}$ in $\Fun(\textup{B} S^1,\Sp)$. So the decompositions seems equivalent.}
First, we claim following fact.\begin{claim}\label{claim:equiv}
    As spectra, $B_m\simeq \textup{Ind}^{S^1}_{C_m} S^{2m}$.
\end{claim}
\begin{claimproof}
Note that we can construct a map \begin{equation*}
    B_m \to \bigoplus_{m\geq 0} B_m \simeq \THH(\Ss[\Omega S^3]) \simeq \bigoplus_{m\geq 0} \textup{Ind}^{S^1}_{C_m} (S^2)^{\otimes m} \simeq \prod_{m\geq 0} \textup{Ind}^{S^1}_{C_m} (S^2)^{\otimes m} \to \textup{Ind}^{S^1}_{C_m} S^{2m}.
\end{equation*}
Now, applying $\Zz\otimes -$, we have maps \begin{equation*}
    \Zz\otimes B_m \to \bigoplus_{m\geq 0}\Zz\otimes B_m \simeq \bigoplus_{m\geq 0} \Zz\otimes \textup{Ind}^{S^1}_{C_m} (S^2)^{\otimes m} \simeq \prod_{m\geq 0}\Zz\otimes  \textup{Ind}^{S^1}_{C_m} (S^2)^{\otimes m} \to \Zz\otimes\textup{Ind}^{S^1}_{C_m} S^{2m}.
\end{equation*}
This map induce equivalence on homotopy groups, by \cite[Lemma 5.5]{kthofthh}, and the fact that $\textup{Ind}^{S^1}_{C_m} S^{2m} \simeq S^{2m}\oplus \Sigma S^{2m}$.
 Now both sides are connective spectra, then we conclude $B_m\simeq \textup{Ind}^{S^1}_{C_m} S^{2m} \simeq S^{2m}\oplus S^{2m+1}$ (this can be seen by filtering $\Ss$ as $\tau_{\geq *}\Ss$, and consider the graded).
\end{claimproof}
 Now, let $W$ be the fiber of the Frobenius map of $F\otimes \THH(\BP\langle n\rangle)$, suppose the $F$-Segal conjecture for $\THH(\BP\langle n\rangle[\Omega S^3])$ is true. Then there is $k_0\in \Zz$, such that $W'(m)\in \Sp_{\leq k_0}$ for all $m\in \Nn,$  where $W'(m)$ is the fiber of following map \begin{equation*}
     F\otimes \THH(\BP\langle n\rangle)\otimes B_m \to F\otimes \THH(\BP\langle n\rangle)^{t C_p}\otimes (B_{mp})^{t C_p}.
 \end{equation*}
This would imply that $W'(m)\otimes \Ss/p \in \Sp_{\leq k_0+1}$, for all $m\in \Nn$.
And $W'(m)\otimes \Ss/p \simeq W\otimes B_m\otimes \Ss/p$, therefore $\Ss/p\otimes W\otimes B_m\in \Sp_{\leq k_0+1}$ for all $m\in \Nn$. However, by Claim \ref{claim:equiv}, we have $\Ss/p\otimes W\otimes B_m\simeq \Sigma^{2m}(\Ss/p\otimes W)\oplus \Sigma^{2m+1} (\Ss/p\otimes W)$.
 By assumption, there is $a\in \Zz$, $\pi_a(\Ss/p\otimes W)\not\simeq 0$, hence $\pi_{a+2m}(\Ss/p\otimes W\otimes B_m)\not\simeq 0$, when $a+2m>k_0$, this contradicts to $W\otimes\Ff_p\otimes B_m\in \Sp_{\leq k_0+1}$. Hence, $F$-Segal conjecture for $\THH(\BP\langle n\rangle[\Omega S^3])$ fails.

 Now, we need to show that $\Ss/p\otimes W$ is not $0$. Suppose $\Ss/p\otimes W\simeq 0$, then $\Ff_p\otimes W\simeq \Zz\otimes \Ss/p\otimes W\simeq 0$. But this will contradicts to following fact, therefore $\Ss/p\otimes W\not\simeq 0$.
 \begin{claim}
     $\Ff_p\otimes W\not\simeq 0$
 \end{claim}
 \begin{claimproof}[Sketch of Proof]
 One could use the argument in the proof of Proposition \ref{pro:segalfails} to prove $W\otimes \Ff_p\not\simeq 0$. One filtered $\THH(\BP\langle n\rangle)^{t C_p}$ using $\desc_{\Ff_p}$, the associated graded is $\THH(\Ff_p[v_0,\ldots,v_n])^{t C_p}$. Then one need to prove modulo $v_0^{i_0}, \ldots, v_{n+1}^{i_{n+1}}$, the $E_{\infty}$-page has negative homotopy groups. Note that $\Ff_p\otimes \THH(\Ff_p[v_0,\ldots,v_n]^{t C_p})\simeq \Ff_p\otimes \Ff_p\otimes_{\Ff_p}\THH(\Ff_p[v_0,\ldots ,v_n])^{t C_p}$. This is literally the same proof with different symbols.    
 \end{claimproof}
 
\end{proof}
However, when we replace $\BP\langle n\rangle$ with $\BP$ the Segal conjecture holds, and even for $\MU$. First recall that in \cite[Theorem 1.1]{segalforMU}, the authors showed that the following Frobenius maps are $p$-adic equivalences.
\begin{align*}
\THH(\MU)\to \THH(\MU)^{t C_p},\\
\THH(\BP)\to \THH(\BP)^{t C_p}.
\end{align*}

Now using the argument before, we can prove the following.
\begin{proposition}\label{pro:segalforgroupring}
    The Frobenius maps \begin{align*}
     \THH(\MU[\Omega S^3]) \to \THH(\MU[\Omega S^3])^{t C_p}\\
        \THH(\BP[\Omega S^3]) \to \THH(\BP[\Omega S^3])^{t C_p}
    \end{align*}
    are $p$-adic equivalences. In general, for $Y$ a space satisfying conditions (1), (2), (3) in Propoistion \ref{pro:segalfailsforgroupring}, then the Frobenius maps \begin{gather*}
       \THH(\MU[\Omega Y]) \to \THH(\MU[\Omega Y])^{t C_p}\\
        \THH(\BP[\Omega Y]) \to \THH(\BP[\Omega Y])^{t C_p} 
    \end{gather*}
    are $p$-adic equivalences.
\end{proposition}
\begin{proof}
We only prove the $\BP$ case, for $\MU$ case, one only need to replace the symbol.
    As seen,  $S^3$ satisfies the conditions (1), (2), (3) in Proposition \ref{pro:segalfailsforgroupring}. Thus the Frobenius map is the product of the following map \begin{equation*}
       \Ss/p\otimes \THH(\BP)\otimes B_m \to\Ss/p \otimes\THH(\BP)^{t C_p}\otimes B_m \to \Ss/p\otimes (\THH(\BP)\otimes B_{mp})^{t C_p},
    \end{equation*}
    where the second map is an equivalence.
    The fiber of $\Ss/p\otimes \THH(\BP)\to \Ss/p\otimes \THH(\BP)^{t C_p}$ is $0$. Hence the first map induce isomorphism on homotopy groups. Thus the total Frobenius map induces isomorphism on homotopy groups. 
\end{proof}

Now, in a same fashion, we can prove the following.
\begin{proposition}\label{pro:segalfortruncated}
Let $\THH(\Ss[x]/x^e)\simeq \bigoplus_{m\geq 0} B_m$ in Proposition \ref{pro:decomposation}.
    The Frobenius maps \begin{align*}
        \THH(\BP[x]/x^e) \to \THH(\BP[x]/x^e)^{t C_p},\\
        \THH(\MU[x]/x^e)\to \THH(\MU[x]/x^e)^{t C_p}.
    \end{align*}
are $p$-adic equivalences.
\end{proposition}
\begin{proof}
We only prove the cases for $\BP$. We can prove the Frobenius map is a product of the following map  \begin{equation*}
    \THH(\BP) \otimes B_m \to \THH(\BP)^{t C_p}\otimes B_m \to (\THH(\BP)\otimes B_{mp})^{t C_p}
\end{equation*}
using a similar argument in Lemma \ref{unstablefro}.
Here, the second map is an equivalence. After tensoring with $\Ss/p$ the fist map has fiber equivalent to $0$ by \cite[Theorem 1.1]{segalforMU}, then the Frobenius map is a $p$-adic equivalence.\footnote{Note that, using the above argument, if one replaces $\THH(\BP)$ with $\THH(\Ss)\simeq \Ss$, then one proves that $\THH(\Ss[x]/x^e)\to \THH(\Ss[x]/x^e)^{t C_p}$ is an $p$-adic equivalence.} 
\end{proof}

Now, we try to investigate the Frobenius map of 
$\THH(\BP\langle n\rangle/v_n^e)$. As seen, we first try to compute the Adams $E_2$-page of $\BP\langle n\rangle/v_n^e$. Note that \begin{equation*}
    \BP\langle n\rangle/v_n^e\otimes \Ff_p\simeq \BP\langle n\rangle/v_n^e\otimes \BP\otimes_{\BP}  \Ff_p.
\end{equation*}
Since $\BP\langle n\rangle/v_n^e$ is complex orientable, then the Tor spectral sequence has the formula \begin{equation*}
    \tor_{*,*}^{\BP_*}(R_*[t_i],\Ff_p).
\end{equation*}
Where $R_*=\BP\langle n,e\rangle_*$ and we have used the notation in Convention \ref{conven3}. Now write $S_* =\BP_*/(p,v_1, \ldots v_{n-1})$, we could identify the derived tensor products \begin{equation*}
    (- \otimes_{\BP_*}^{\Ll} S_*)\otimes_{S_*}^{\Ll} \Ff_p \simeq - \otimes_{\BP_*}^{\Ll} \Ff_p.
\end{equation*}
Hence we get the Cartan\textendash Elienberg (Grothendieck) spectral sequence 
\begin{equation*}
    \tor_{*,*}^{S_*}(\tor_{*,*}^{\BP_*}(R_*[t_i], S_*),\Ff_p) \Longrightarrow \tor_{*,*}^{\BP_*}(R[t_i],\Ff_p).
\end{equation*}
Note that the image of $p,v_1,\ldots, v_{n_1}$ through the map $\BP_*\to \BP_*(\BP) \to R_*(\BP) \simeq R_*[t_i]$ forms a regular sequence, hence the higher tor groups of $\tor^{\BP_*}(R_*[t], S_*)$ vanish. And we have $R_*[t]\otimes_{\BP_*} S_* \simeq \Ff_p[v_n]/v_n^e [t_i],$ thus the spectral sequence has the formula \begin{equation*}
    \Ff_p[v_n]/v_n^e[t_i] \otimes \Lambda(x_n, x_{n+1}, \ldots).
\end{equation*}
This should be a sub-algebra of $(\Ff_p)_*(\Ff_p\otimes \Ss[v_n]/v_n^e)$, where $\Ss[v_n]$ is a free graded $\Ee_2$-algebra, and the quotient of $v_n^e$ is a truncation with respect to the weight, see \cite[Appendix B.0.6]{HW22}. 
Although, we could not compute the Adams $E_2$-page of $\BP\langle n,e\rangle$ for $e>1$ \footnote{It seems that the Adams $E_2$-page should be $\Ff_p[v_n]/v_n^e[v_0, \ldots, v_{n-1}]$.}. But we could still investigate the Frobenius map of $\THH(\Ff_p\otimes\Ss[v_0]\otimes\cdots\otimes \Ss[v_{n-1}]\otimes \Ss[v_n]/v_n^e])$. In particular, one could ask the following question.
\begin{question}
Is there a $k_0\in \Zz$, such that $\forall k>k_0$ the map \begin{equation*}
   \THH_k(\Ff_p\otimes \Ss[v_0]\otimes\cdots\otimes \Ss[v_{n-1}]\otimes \Ss[v_n]/v_n^e) \to \THH_k(\Ff_p\otimes \Ss[v_0]\otimes\cdots\otimes \Ss[v_{n-1}]\otimes \Ss[v_n]/v_n^e)^{t C_p}
\end{equation*}
is an isomorphism.    
\end{question}
\begin{remark}
To investigate the above question, one may need to understand the structure of $\THH(\Ss[v_n]/v_n^e)$. A natrual expectation is that a similar decomposition exists.  Alternatively, it is necessary to analyze the behavior of the Frobenius \begin{equation*}
    \THH(\Ff_p\otimes \Ss[v_n]/v_n^e)\to \THH(\Ff_p\otimes \Ss[v_n]/v_n^e)^{t C_p},
\end{equation*}
as this constitutes the only remain unknown part.
\end{remark}
Next, we analyze the behavior of the Frobenius map of $\THH(\Lambda_{\BP\langle n\rangle}(x))$.
According to construction in Section \ref{sec2}, $\Lambda_{\BP\langle n\rangle} =\BP\langle n\rangle \otimes \Lambda_{\Ss}(x)$. It is known that $\pi_*(\Lambda_{\Ss}(x))\simeq \Lambda_{\pi_*(\Ss)}(\sigma x)$, and it is easy to construct a map from $\Ss\oplus \Sigma\Ss$ to $\Lambda_{\Ss_*}(\sigma x)$, namely taking the homotopy classes $1, \sigma x$. This map will induce isomorphisms between the homotopy groups of $\Ss\oplus \Sigma\Ss$ and homotopy groups of $\Lambda_{\Ss}(x)$. Therefore we have an equivalence of spectra ($\Ss$-modules) \begin{equation*}
    \Lambda_{\Ss}(x) \simeq \Ss\oplus \Sigma\Ss.
\end{equation*}
And the same thing holds for $\Lambda_{\BP\langle n\rangle} (x)$. Hence by \cite[Remark 5.7]{HLS24}, we have the following decomposition \begin{equation*}
    \THH(\Lambda_{\BP\langle n\rangle}(x))\simeq \THH(\BP\langle n\rangle)\oplus \bigoplus_{k\geq 1} \THH(\BP\langle n\rangle)\otimes\Hom((S^1/C_k)_+, (\Sigma^2 \BP\langle n\rangle)^{\otimes k}).
\end{equation*}
Because $\THH(\BP\langle n\rangle)$ is connective, then the summand $\THH(\BP\langle n\rangle)\otimes \Hom((S^1/C_k)_+, (\Sigma^2\BP\langle n\rangle)^{\otimes k})$ is $2k-1$ connected, so when $k\to \infty$, the connectivity of the summand tends to infinity. Thus we have the following equivalences.
\begin{equation*}
    \THH(\Lambda_{\BP\langle n\rangle}(x))\simeq \THH(\BP\langle n\rangle) \oplus \prod \Hom((S^1/C_k)_+, (\Sigma^2\BP\langle n\rangle)^{\otimes k}),
\end{equation*}
\begin{equation*}
    \THH(\Lambda_{\BP\langle n\rangle}(x))^{t C_p}\simeq \THH(\BP\langle n\rangle)^{t C_p} \oplus \prod_{k\geq 1} (\THH(\BP\langle n\rangle)\otimes \Hom((S^1/C_k)_+, (\Sigma^2\BP\langle n\rangle)^{\otimes k}))^{t C_p}.
\end{equation*}
This is because the direct sum is equivalent to the product, and then taking homotopy orbit and homotopy fixed point would commute with colimits and limits.
So, as seen before, the Frobenius map is the product of the following maps.
\begin{enumerate}
    \item \begin{equation*}
        \THH(\BP\langle n\rangle)\to \THH(\BP\langle n\rangle)^{t C_p}.
        \end{equation*}
        \item \begin{equation*}
        \THH(\BP\langle n\rangle)\otimes \Hom((S^1/C_k)_+, (\Sigma^2\BP\langle n\rangle)^{\otimes k})\to
           (\THH(\BP\langle n\rangle)\otimes \Hom((S^1/C_k)_+, (\Sigma^2\BP\langle n\rangle)^{\otimes kp}))^{t C_p}.
        \end{equation*}
\end{enumerate}
As seen before, one could expect that for $k$ sufficient large the fiber of type (2) map should have highly enough connectivity. So, one may ask the following question.
\begin{problem}
According to the above discussion, let $F$ be a type $n+2$
$p$-local finite complex.
Assume that the fiber of the Frobenius map 
\begin{equation*}
\xymatrix{
F\otimes \THH(\BP\langle n\rangle)\otimes \Hom((S^1/C_k)_+, (\Sigma^2 \BP\langle n\rangle)^{\otimes k})) \ar[d] \\
~ F\otimes(\THH(\BP\langle n\rangle)\otimes \Hom((S^1/C_k)_+, (\Sigma^2\BP\langle n\rangle)^{\otimes kp}))^{t C_p}
}
\end{equation*}
has connectivity $f(k)$. Do we have $\lim_{k\to \infty} f(k)=\infty$? In particular, do we have $\lim_{n\to \infty} f(p^n) =\infty$?  
\end{problem}
Or, one could ask the following question.
\begin{problem}
Let  $F$ be any $p$-local finite complex of type $n+2$. Is \begin{equation*}
    \text{
    \framebox{$F$-Segal conjecture for $\THH(\Lambda_{\BP\langle n\rangle}(x))$}
    }
\end{equation*}
true?    
\end{problem}

\section{Descent spectral sequence}\label{sec5}
In this section, we study the descent spectral sequences (and even filtration) which computes $\THH$ of algebras which constructed in Section \ref{sec2}. And also, we would like to investigate the canonical vanishing problem.

Note that $\THH(\MU[x])\to \MU[x]$ is evenly free. We need to show that for any map of $\Ee_1$-rings $\THH(\MU[x])\to A$, where $A$ is even,  the push-out $\MU[x]\otimes_{\THH(\MU[x])}A$ is even and free over $A$. Recall that 
$\MU[x]\otimes\MU[x]$ has homotopy as $(\MU_*\MU)[x,y]\simeq \MU_*[x,x'][x_i|i\in \Nn]$. Therefore $\THH(\MU[x])_*\simeq \Lambda_{\MU_*[x]}(\sigma x_i, \sigma x)$, which is odd. Hence the classes $\sigma x_i, \sigma x$ maps to $0$ in $\pi_* A$, therefore the Tor spectral sequence computes $\MU[x]\otimes_{\THH(\MU[x])}A$ is $A_*\otimes_{\MU_*[x]}\Gamma_{\MU_*[x]}(\sigma^2 x_i, \sigma^2 x; i\in \Nn)$ which is concentrated on even degrees, hence the push-out $\MU[x]\otimes_{\THH(\MU[x])}A$ is even. 
Now, combining with Proposition \ref{THHtruncated}, we have following fact.
\begin{proposition}
Let $\motfil$ be the motivic filtration, then we have the following equivalence of filtered spectra.
\begin{equation*}
    \motfil (\THH(\BP\langle n\rangle[x]/x^e)) \simeq \lim_{\triangle} \tau_{\geq 2*} (\THH(\BP\langle n\rangle[x]/x^e)\otimes \MU[x]^{\otimes_{\THH(\BP\langle n\rangle[x]/x^e/\MU[x])} \bullet+1})
\end{equation*}
\end{proposition}
\begin{remark}
    In general, if $A$ is an $\Ee_3$-algebra, and $B$ is an $\Ee_2$-$A$-algebra, then we could view $A,B$ as filtered algebras via Whitehead tower. And form the cosimplicial filtered object \begin{equation*}
        [n] \mapsto B^{\otimes_A [n]},
    \end{equation*}
then the geometrical realization of this cosimplicial object captures some information about $A$. Then one can use the Adams tower trick, and apply Hopf algebroid method. In \cite{LW22}, the authors use the map $\THH(\mathcal{O}_K/\Ss_{W(k)}) \to \THH(\mathcal{O}_K/\Ss_{W(k)}[z])$ to get a descent tower. However, to compute $\TC(\mathcal{O}_{K})^{\wedge}_p$, one need to show the mentioned two relative $\THH$ are actually cyclotomic spectra, i.e. to show the bases are cyclotomic bases. In the setting of motivic (even) filtration, one do not need cyclotomic bases, since the motivic filtration for $\THH$ is automatically cyclotomic. 
\end{remark}
As in \cite{HW22}, we make use of the following notations.
\begin{convention}
 If $A$ is an $\Ee_2$-algebra, and $M$ is a right $A$-module, hence $M$ can be viewed as a $A$-bi module, and we have following equivalences.
 \begin{equation*}
     \THH(A;M)=M\otimes_{A\otimes A^{\textup{op}}} A\simeq M\otimes_{A} \THH(A).
 \end{equation*}
 Hence we will have the $\THH(\BP\langle n\rangle[x]/x^e; \Ff_p[x]/x^e) \simeq \Ff_p[x]/x^e \otimes_{\BP\langle n\rangle[x]/x^e} \THH(\BP\langle n\rangle[x]/x^e).$
\end{convention}
In the following, we assume that $\BP\langle n\rangle$ equipped with an $\Ee_{\infty}$-algebra structure, for instance, $n\in \{-1,0,1\}$, or $p=3, n=2$, see \cite{BR05}, \cite{HW22}.
Using the calculation in \cite[\S 6]{HW22}, we have the following result. 
\begin{proposition}\label{pro:descentE2}
    The $E_2$-page of descent spectral sequence for \begin{equation*}
        \THH(\BP\langle n\rangle[x]/x^e; \Ff_p[x]/x^e) \to \THH(\BP\langle n\rangle [x]/x^e/\MU[x]; \Ff_p[x]/x^e)
    \end{equation*}
is  \begin{equation*}
    E_2=\Ff_p[\sigma^2 v_{n+1}] \otimes \Lambda(\sigma t_1, \sigma t_2,\ldots \sigma t_{n+1})\otimes\Lambda_C (\sigma (x_1- x_2))
\end{equation*}
where $\sigma^2 v_{n+1}, \sigma t_i$ are the same elements appeared in \cite[Proposition 6.1.6]{HW22}, and $C=\HH_*(\Ff_p[x]/x^e/\Ff_p[x])$. 
When $p|e$, then spectral sequence collapses at $E_2$-page.
\end{proposition}
\begin{proof}
We write $A=\THH(\BP\langle n\rangle[x]/x^e;\Ff_p[x]/x^e), B=\THH(\BP\langle n\rangle[x]/x^e/\MU[x];\Ff_p[x]/x^e), A'=\THH(\BP\langle n\rangle;\Ff_p), B'=\THH(\BP\langle n\rangle/\MU; \Ff_p)$. Then we viewed $A, B$ as filtered algebras, via the Whitehead filtration, we could form the cosimplicial filtered spectrum \begin{equation*}
    [n] \mapsto B^{\otimes_A [n]}.
\end{equation*}
Note that we have equivalences\footnote{It seems that the equivalence above could be only happen in $\Sp$, however, if $n\in\{-1,0,1\},$ then $\BP\langle n\rangle$ are $\Ee_{\infty}$-algebra, thus the equivalence would happen in $\CAlg$.} \begin{equation*}
    B\simeq \Ff_p[x]/x^e \otimes_{\BP\langle n\rangle\otimes \Ss[x]/x^e} (\THH(\BP\langle n\rangle/\MU) \otimes\THH(\Ss[x]/x^e)/\Ss[x]) \simeq B'\otimes \THH(\Ss[x]/x^e/\Ss[x]),
\end{equation*}
similarly, we have equivalence \begin{equation*}
    A\simeq A'\otimes \THH(\Ss[x]/x^e)\simeq A'\otimes_{\Ff_p} \Ff_p\otimes \THH(\Ss[x]/x^e).
\end{equation*}
So, $B\otimes_A B \simeq B'\otimes_{A'} B' \otimes_{\Ff_p} \Ff_p \otimes \THH(\Ss[x]/x^e/\Ss[x]) \otimes_{\THH(\Ss[x]/x^e)} \THH(\Ss[x]/x^e/\Ss[x])\simeq B'\otimes_{A'} B'\otimes_{\Ff_p} \Ff_p \otimes\THH(\Ss[x]/x^e/\Ss[x_1, x_2])$. Hence using Lemma \ref{dividedpower}, we have \begin{equation*}
  \pi_*(B\otimes_A B) \simeq \pi_*(B'\otimes_{A'} B')\otimes_{\Ff_p} \Gamma_{\Ff_p[x]/x^e}(\sigma^2 x^e, \sigma^2 (x_1-x_2)),   
\end{equation*}and \begin{equation*}
 \pi_* B\simeq \pi_* B'\otimes_{\Ff_p} \HH_*(\Ff_p[x]/x^e/\Ff_p[x])),   
\end{equation*} thus $\pi_*(B\otimes_A B)$ is flat over $\pi_* B$. 
Now we could compute the Ext group \begin{equation*}
    \ext^*_{\sum} (\pi_* B, \pi_* B),
\end{equation*}
where we denote $\sum:=\pi_*(B\otimes_A B) \simeq \pi_* (B'\otimes_{A'} B')\otimes_{\Ff_p} \Gamma_{\Ff_p[x]/x^e} (\sigma^2 x^e, \sigma^2 (x_1-x_2)).$  Therefore the Ext group is computed by \begin{equation*}
    \ext_{\sum}(\pi_* B,\pi_* B) \simeq \ext_{\pi_*(B'\otimes_{A'} B')}(\pi_* B', \pi_*B')\otimes\ext_D(C,C),
\end{equation*}
where $C=\Gamma_{\Ff_p[x]/x^e}(\sigma^2 x^e), D=\Gamma_{\Ff_p[x]/x^e}(\sigma^2 x^e, \sigma^2 (x_1-x_2))$. \footnote{Note that $D$ is relative injective, and $D$ is free as $C$-module.}
Finally, the Ext group is equivalence to \begin{equation*}
    \Ff_p[\sigma^2 v_{n+1}]\otimes_{\Ff_p} \Lambda_{\Ff_p}(\sigma t_1, \sigma t_2,\ldots, \sigma t_{n+1})\otimes_{\Ff_p}\Lambda_C(\sigma (x_1-x_2)).
\end{equation*}
According to the result in \cite{HM97}, we have \begin{equation*}
    \HH_*(\Ff_p[x]/x^e/\Ff_p)=\begin{cases}
        \Ff_p[x]/x^e, ~ *=0;\\
        x\Ff_p[x]/x^e ~ *>0 \text{~and even};\\
        \Ff_p[x]/x^{e-1} ~ *\text{~odd, and} ~ (e,p)=1;\\
        \Ff_p[x]/x^{e-1}\oplus \Ff_p\{x^{e-1}\}, ~ *\text{~odd, and~} p|e.
    \end{cases}
\end{equation*}
Therefore when $p|e$, then $E_2$-page of the descent spectral sequence has the size same as $\pi_* A$, then in this case the spectral sequence collapses.
\end{proof}
\begin{remark}
    From the above proposition, and combining with the proof of canonical vanishing theorem in \cite[\S 6]{HW22}, one could not directly prove the canonical vanishing theorem for $\THH(\BP\langle n\rangle[x]/x^e)$, because the $E_2$-page above is not a finitely generated $\Ff_p[\sigma^2 v_{n+1}]$-module.
\end{remark}
We can prove the following by a thick sub-category argument.
\begin{proposition}
Suppose $p|e$, and let $F$ be a type $n+1$ finite complex, and $F$ has a homotopy ring structure, then the elements in the kernel of the map \begin{equation*}
    F_*(\THH(\BP\langle n\rangle[x]/x^e)) \to F_*(\THH(\BP\langle n\rangle[x]/x^e/\MU[x]))
\end{equation*}
are nilpotent.
\end{proposition}
Let $M$ be the finite $p$-local $\Ee_1$-ring constructed in \cite[Proposition 6.2.1]{HW22}, note that we have $M\otimes \THH(\BP\langle n\rangle[x]/x^e) \simeq\prod_{m\geq 0} M\otimes \THH(\BP\langle n\rangle)\otimes B_m$, and $B_0\simeq \Ss$. We try to show following. 
\begin{proposition}
    There is an element $z\in \pi_*(M\otimes \THH(\BP\langle n\rangle[x]/x^e))$ such that the following holds.
    \begin{enumerate}
        \item $z$ is central.
        \item $z$ maps to a power of $\sigma^2 v_{n+1}$ inside $\pi_*(M\otimes \THH(\BP\langle n\rangle[x]/x^e/\MU[x]))$.
        \item There is a $m\geq 0$, such that $t^m z$ in the $E_2$-page of homotopy fixed point spectral sequence detects the image of an central $v_{n+1}$-element from $\pi_* M$ inside $\pi_*M\otimes\THH(\BP\langle n\rangle)^{h S^1}$, and hence detects the image of $v_{n+1}$-element in $\pi_*(M\otimes \TC^-(\BP\langle n\rangle[x]/x^e))$. 
        \item $\pi_* M\otimes \THH(\BP\langle n\rangle[x]/x^e)$ is a $\Zz_{(p)}[z]$-module.
    \end{enumerate}
\end{proposition}
\begin{proof}
Note that we have \begin{equation*}
        M\otimes\THH(\BP\langle n\rangle[x]/x^e/\MU[x])\simeq M\otimes\BP\langle n\rangle[x]/x^e\otimes_{\BP\langle n\rangle[x]/x^e} \THH(\BP\langle n\rangle[x]/x^e/\MU[x]),
    \end{equation*}
and $M\otimes \BP\langle n\rangle[x]/x^e$ is a sum of shift of $\Ff_p[x]/x^e$,  and \begin{equation*}
 \THH(\BP\langle n\rangle[x]/\MU[x];\Ff_p[x]/x^e)\simeq \THH(\BP\langle n\rangle/\MU;\Ff_p)\otimes_{\Ff_p} \HH(\Ff_p[x]/x^e/\Ff_p[x]).   
\end{equation*}
So, one could find a such lift $z$ of $\sigma^2 v_{n+1}$ in $\pi_*(M\otimes \THH(\BP\langle n\rangle[x]/x^e/\MU[x]))$ satisfies the condition (1), (2), (3), (4), by \cite[Lemma 6.3.5]{HW22}.
\end{proof}
\section{Computation for height one case}\label{sec6}
In height $-1$ case, the computation of $\TC(\BP\langle n\rangle[x]/x^e)$ was presented in \cite{HM97}. We will use the strategy and the computations in \cite{AR02} to carry out some computations for $\TC(\ell[x]/x^e)$, where we use $\ell$ to denote $\BP\langle 1\rangle$.

By Proposition \ref{pro:decomposation}, we have the following equivalence \begin{equation*}
    \THH(\ell[x]/x^e)\simeq \bigoplus_{m\geq 0}\THH(\ell)\otimes B_m\simeq \prod_{m\geq 0} \THH(\ell)\otimes B_m,
\end{equation*}
and since the summand has connectivity tends to $\infty$, we therefore have following equivalences
\begin{gather}\label{equ:decom1}
F\otimes\TC^-(\ell[x]/x^e)\simeq \prod_{m\geq 0} F\otimes (\THH(\ell) \otimes B_m)^{h S^1},\\
\label{equ:decom2}
F\otimes \TP(\ell[x]/x^e) \simeq \prod_{m\geq 0} F\otimes (\THH(\ell) \otimes B_m)^{t S^1}.  
\end{gather}
Where $F$ is a $p$-local type $2$ finite complex, for instance, $F=V(1)=\Ss/(p, v_1)$, (we may assume $p>3$ in the remaining context). We know that 
\begin{equation*}
  F\otimes \TC(\ell[x]/x^e)\simeq F\otimes \TC(\ell[x]/x^e,\Zz_p)\simeq \fib(\can_p-\varphi_p: F\otimes \TC^-(\ell[x]/x^e,\Zz_p) \to F\otimes \TP(\ell[x]/x^e,\Zz_p)).
\end{equation*}
Recall a theorem that allow us to compute (\ref{equ:decom1}) and (\ref{equ:decom2}).
\begin{proposition}[{\cite[Proposition 3]{cusp}}]\label{pro:Tatefix}
Let $G$ be a compact Lie group, $H\subset G$ be a closed subgroup, let $\lambda=T_H(G/H)$ be the tangent space at $H=eH$ with adjoint left $H$-action. Let $S^{\lambda}$ denote the compactification of $\lambda$, then for any $G$-spectrum $X$, we have the following equivalences which are natural.
\begin{gather*}
    (X\otimes (G/H)_+)^{h G}\simeq (X\otimes S^{\lambda})^{h H},\\
    (X\otimes (G/H)_+)^{t G}\simeq (X\otimes S^{\lambda})^{t H}.
\end{gather*}
\end{proposition}
In our case, $X$ is $\THH(\ell)\otimes S^{V_d}$, $H=C_m$, and $G=S^1$, hence applying Proposition \ref{pro:Tatefix}, Proposition \ref{pro:decomposation}, we have the following equivalences.
\begin{gather*}
F\otimes\TC^-(\ell[x]/x^e, \Zz_p)\simeq \prod_{e\nmid m} (\Sigma^{[\frac{m-1}{e}]+1}(F\otimes\THH(\ell)^{h C_p^{v_p(m)}})) \times \prod_{k=m/e\geq 1}\cofib (\Sigma^{[\frac{m-1}{e}]+1} V_{v_p(e),v_p(m)});\\
F\otimes\TP(\ell[x]/x^e, \Zz_p)\simeq \prod_{e\nmid m} (\Sigma^{[\frac{m-1}{e}]+1}(F\otimes\THH(\ell)^{t C_p^{v_p(m)}})) \times \prod_{k=m/e\geq 1}\cofib (\Sigma^{[\frac{m-1}{e}]+1} V_{v_p(e),v_p(m)}).
\end{gather*}
By Corollary \ref{cor:TRboundedfails}, $V(2)\otimes \TR(\ell[x]/x^e)$ is not bounded, so one would not expect that $V(2)\otimes\TC(\ell[x]/x^e)$ is bounded. However, this is still not clear.
According to the above decomposition, we will mainly focus on the easy part. First recall the calculations in \cite{AR02}, we have following.
\begin{proposition}[{\cite[Proposition 6.2, 6.5]{AR02}}]\label{pro:gradedofSS}
   The associated graded of the Tate spectral sequence computing $V(1)_*(\THH(\ell)^{t C_{p^{n+1}}})$  is given by 
   \begin{align*}
      \hat{E}^{\infty}(C_{p^{n+1}})=\Lambda_{\Ff_p}(\lambda_1,\lambda_2)\otimes P_{r(2n)+1}(t\mu)\otimes P(t^{p^{2n+2}}, t^{-p^{2n+2}})\oplus \\ 
      \bigoplus_{k=3}^{2n+2}\Lambda(u_{n+1}, \lambda_{[k]}')\otimes P_{r(k-2)}(t\mu)\otimes \Ff_p\{\lambda_{[k]}t^i\mid v_p(i)=k-1\}. 
   \end{align*}
   The associated graded of the homotopy fixed pointed spectral sequence computing $V(1)_*(\THH(\ell)^{h C_{p^n}})$ maps by a $(2p-2)$-coconnected map to 
   \begin{align*}
   \mu^{-1}E^{\infty} (C_{p^n})=\Lambda(\lambda_1,\lambda_2)\otimes P_{r(2n)+1}(t\mu)\otimes P(\mu^{p^{2n}}, \mu^{-p^{2n}})\oplus \\
   \bigoplus_{k=1}^{2n}\Lambda(u_n,\lambda'_{[k]})\otimes P_{r(k)}(t\mu)\otimes\Ff_p\{\lambda_{[k]} \mu^j| v_p(j)=k-1\}.     
   \end{align*}
The Frobenius map sends $\lambda_1\mu^{p^{2n-2}}, \lambda_2\mu^{p^{2n-1}}, \mu^{p^{2n}}$ respectively to $\lambda_1 t^{-p^{2n}}, \lambda_2 t^{-p^{2n+1}}, t^{-p^{2n+2}}$.
\end{proposition}

\begin{remark}
  We have an equivalence \begin{equation*}
    F\otimes \TC(\ell[x]/x^e)\simeq \fib(\can_p-\varphi: F\otimes \TC^-(\ell[x]/x^e) \to F\otimes \TP(\ell[x]/x^e)).
\end{equation*}  A summand of this is \begin{equation*}
    \fib (\can_p-\varphi_p: \prod_{e\nmid m} (\Sigma^{[\frac{m-1}{e}]+1}(F\otimes\THH(\ell)^{h C_p^{v_p(m)}})) \to \prod_{e\nmid m} (\Sigma^{[\frac{m-1}{e}]+1}(F\otimes\THH(\ell)^{t C_p^{v_p(m)}})))
\end{equation*}
We know that $v_2$ is detected by $t\mu$ in $V(1)_*\TC^-(\ell)$. Let $A=\Lambda_{\Ff_p}(\lambda_1,\lambda_2)\otimes P(t^{p^{2n+2}})$, $B=\bigoplus_{k=3}^{2n+2}\Lambda(u_{n+1}, \lambda_{[k]}')\otimes \Ff_p\{\lambda_{[k]}t^i\mid v_p(i)=k-1\}$, then
$\hat{E}^{\infty}(C_{p^{n+1}})/(t\mu)=A\oplus B$. We know that both canonical map and Frobenius map send $\lambda_1, \lambda_2$ to the same name elements in the target. Hence these classes will contribute to the homotopy groups on $V(1)_*\TC(\ell[x]/x^e)/(v_2)$, i.e. we  have following fact. 
\end{remark}
\begin{proposition}\label{pro:TClxe}
Let $e>1$.
 $\pi_{[\frac{m-1}{e}]+2p-1} (V(1)\otimes \TC(\ell[x]/x^e)/(v_2))\not\simeq 0$, and one of the generator is $\lambda_1$ as $\Ff_p$-vector space, and $\pi_{[\frac{m-1}{e}]+2p^2-1} (V(1)\otimes \TC(\ell[x]/x^e)/v_2)\not\simeq 0$, one of the generator is $\lambda_2$, and $\pi_{[\frac{m-1}{e}]+2p^2-2} (V(1)\otimes \TC(\ell[x]/x^e)/v_2)\not\simeq 0$, one of the generator is $\partial \lambda_2$. Hence, when $m\to \infty$, we get that $\TC(\ell[x]/x^e)/(p, v_1, v_2)$ is not bounded.
Therefore the spectrum  $\TC(\ell[x]/x^e)/(p,v_1, v_2)$ is not bounded, hence $\K((\ell[x]/x^e)^{\wedge}_p)^{\wedge}_p/(p,v_1, v_2)$ is not bounded.
\end{proposition}

\bibliographystyle{alpha}
\bibliography{ref}

\end{document}